\documentclass[a4paper]{amsart}

\usepackage[latin1]{inputenc}
\usepackage[T1]{fontenc}
\usepackage[english]{babel}
\usepackage{amsmath}
\usepackage{verbatim}
\usepackage{amsfonts}
\usepackage[all]{xy}
\usepackage{amssymb}
\usepackage{MnSymbol}

\usepackage{hyperref}
\usepackage{mathrsfs}
\usepackage{color}
\usepackage{enumerate}

\newcommand{\incl}[1][r]{\ar@<-0.2pc>@{^(-}[#1] \ar@<+0.2pc>@{-}[#1]}
\newcommand{\xdashrightarrow}[2][]{\ext@arrow 0359\rightarrowfill@@{#1}{#2}}

\newcommand{\Ker}{\operatorname{Ker}}

\newcommand{\Spec}{\operatorname{Spec}}
\newcommand{\Pic}{\operatorname{Pic}}
\newcommand{\Div}{\operatorname{Div}}
\renewcommand{\div}{\operatorname{div}}
\newcommand{\Cl}{\operatorname{Cl}}
\newcommand{\Bl}{\operatorname{Bl}}
\newcommand{\Sing}{\operatorname{Sing}}
\newcommand{\Frac}{\operatorname{Frac}}
\newcommand{\Face}{\operatorname{Face}}
\newcommand{\Supp}{\operatorname{Supp}}
\renewcommand{\deg}{\operatorname{deg\ }}
\def\Vert{{\rm Vert}}
\def\Ray{{\rm Ray}}
\def\PL{{\rm PL}}

\newcommand{\tor}{\rm{tor}}
\newcommand{\SL}{\rm SL}
\newcommand{\PSL}{\rm PSL}

\newcommand{\PGL}{\rm PGL}
\newcommand{\Hom}{\rm Hom}
\newcommand{\ord}{\rm ord}
\newcommand{\Sch}{\rm Sch}

\def\PP{{\mathbb P}}
\def\BB{{\mathcal B}}
\def\ZZ{{\mathbb Z}}
\def\NN{{\mathbb N}}
\def\QQ{{\mathbb Q}}

\def\G{{\mathbb G}}
\def\D{{\mathfrak D}}
\def\F{{\mathcal F}}

\def\AA{{\mathbb A}}

\def\Cc{{\mathcal C}}

\def\ES{\Sigma}
\def\SC{{\rm Sch}(Z)}
\def\SG{{\rm Sch}_{G}(Z)}
\def\E{\mathscr{E}}
\def\AS{\mathbb{A}_{\star}}
\def\dc{dec}
\def\Prin{{\rm Prin}}

\theoremstyle{plain}
\newtheorem{theorem}{Theorem}[section]
\newtheorem{lemma}[theorem]{Lemma}
\newtheorem{proposition}[theorem]{Proposition}
\newtheorem{corollary}[theorem]{Corollary}
\newtheorem*{theorem*}{Theorem}

\theoremstyle{definition}
\newtheorem{definition}[theorem]{Definition}
\newtheorem{example}[theorem]{Example}

\newtheorem*{notation}{Notation}
\newtheorem*{Ack}{Acknowledgements}

\theoremstyle{remark}
\newtheorem{remark}[theorem]{Remark}

\address{Instituto de Ciencias Matematicas (ICMAT)\\ Campus Cantoblanco \\ Madrid, Spain}
\email{langlois.kevin18@gmail.com} 

\address{Fachbereich Physik, Mathematik und Informatik\\ Johannes Gutenberg -- Universit\"at Mainz \\ Mainz, Germany}
\email{rterpere@uni-mainz.de}

\begin{document}

\title[On the geometry of horospherical varieties of complexity one]{On the geometry of normal horospherical $G$-varieties of complexity one}

\author{Kevin Langlois}
\author{Ronan Terpereau}

\begin{abstract}
Let $G$ be a connected simply-connected reductive algebraic group. In this article, we consider the normal algebraic varieties equipped with a horospherical $G$-action such that the quotient of a $G$-stable open subset is a curve. Let $X$ be such a $G$-variety. Using the combinatorial description of Timashev, we describe the class group of $X$ by generators and relations and we give a representative of the canonical class. Moreover, we obtain a smoothness criterion for $X$ and a criterion to determine whether the singularities of $X$ are rational or log-terminal respectively. 
\end{abstract}

\maketitle
\thispagestyle{empty}


\tableofcontents

\section*{Introduction}

The varieties and the algebraic groups that we consider are defined over an algebraically closed field $k$ of characteristic zero. Let $G$ be a connected simply-connected reductive algebraic group and let $B \subset G$ be a Borel subgroup. The aim of this article is to describe certain geometric properties of a family of $G$-varieties: the normal horospherical $G$-varieties of complexity one. Among other results, we obtain explicit criteria to characterize the singularities of these varieties, we describe their class group by generators and relations, and we give an explicit representative of their canonical class.

Recall that the \emph{complexity} of a $G$-variety $X$ is the transcendence degree of the field extension $k(X)^{B}$ over $k$, where $k(X)^{B}$ denotes the field of $B$-invariant rational functions on $X$; see \cite{LV,Vi}. Since all the Borel subgroups of $G$ are conjugated, this notion does not depend on the choice of $B$. Geometrically, the complexity is the codimension of a general $B$-orbit. For instance, the $G$-varieties of complexity zero contain an open $B$-orbit. Those which are normal are called \emph{spherical varieties}; see \cite{Kn, Pez, Hur11, Ti3, Per} for more information.

The study of the $G$-varieties of complexity one is the next step (after the spherical case) towards the classification of normal $G$-varieties. 
Many important examples of $G$-varieties of complexity one motivate this study:
\begin{itemize}
 \item[$\bullet$] The normal $T$-varieties of complexity one, where $T$ is a torus; see \cite{KKMS, Ti2} for a combinatorial description, \cite{La} for a generalization over an arbitrary field, \cite{FZ} for the case of surfaces, and \cite{AH, AHS,AIPSV} for higher complexity.
\item[$\bullet$] The homogeneous spaces $G/H$ of complexity one, where $H$ is a connected reductive closed subgroup, are described in \cite{Pan92,AC}.
\item[$\bullet$] The normal three-dimensional affine $\SL_2$-varieties with a two-dimensional general orbit are studied in \cite{Arz97}.
\item[$\bullet$] The embeddings of ${\rm SL_{2}}/K$ into normal $\SL_2$-varieties, where $K$ is a finite subgroup, are studied in \cite{Pop}, \cite[\S 9]{LV}, \cite{Mos1}, and \cite{Mos2}. More generally, see \cite[\S 16.5]{Ti3} for a classification of the embeddings of $G/K$ where $G$ is a semisimple group of rank $1$ and $K \subset G$ is a finite subgroup. 
\item[$\bullet$] An example from classical geometry: Let $T \subset \SL_3$ be the subgroup of diagonal matrices. The homogeneous space ${\rm SL_3}/T$ can be identified with the set of ordered triangles on $\PP^2$. The embeddings of ${\rm SL_3}/T$ are studied in \cite{War} and \cite[16.5]{Ti3}.
\end{itemize}

A combinatorial description of normal $G$-varieties of complexity one is obtained in \cite{Ti}. This description is inspired by the Luna--Vust theory of embeddings of homogeneous spaces $G/H$ into normal varieties; see \cite{LV}. 

A $G$-action is called \emph{horospherical} if the isotropy group of any point contains a maximal unipotent subgroup of $G$. Therefore a homogeneous space $G/H$ is horospherical if and only if $H$ contains a maximal unipotent subgroup; such an $H$ is called a \emph{horospherical subgroup} of $G$. It follows from the Bruhat decomposition of $G$ that every horospherical $G$-homogeneous space is spherical, and thus a general $G$-orbit of a horospherical $G$-variety of complexity one has codimension one. We will recall in \S \ref{pasquier} how the set of horospherical subgroups of $G$ containing the unipotent radical of $B$ can be described from the set of simple roots of $G$.

The $G$-equivariant birational class of a horospherical $G$-variety $X$ is determined by the invariant field $k(X)^{G}$ and by the isotropy subgroup $H$ of a general point. Indeed, if $r$ denotes the complexity of the $G$-variety $X$, then by \cite[Satz 2.2]{Kn1} there exist an $r$-dimensional variety $C$ and a $G$-equivariant birational map  
$$\phi : X \dashrightarrow Z:=C\times G/H,$$ 
where $G$ acts on $Z=C\times G/H$ by translation on $G/H$. The map $\phi$ induces field isomorphisms $k(X)\simeq \Frac \left(k(C)\otimes_{k}k(G/H)\right)$ and $k(C)\simeq k(X)^{G}$. If $r=1$, then we can assume that $C$ is a smooth projective curve.

This article is structured as follows:
In the first part, we set up our framework by explaining the combinatorial description of normal horospherical $G$-varieties of complexity one (following Timashev). Let us be more precise. In \S \ref{pasquier} we describe the horospherical $G$-homogeneous spaces. In \S \ref{subsec12} we recall the definition of the scheme of geometric localities $\SG$ whose normal separated $G$-stable open subsets of finite type, also called $G$-models of $Z$, are the normal $G$-varieties $G$-birational to $Z$. We also introduce the notion of a chart, which is a $B$-stable affine open subset of $\SG$, and of a germ, which is a (proper) $G$-stable closed subvariety of $\SG$. Then we consider the $B$-stable divisors on $G/H$, also called colors (of $G/H$), and we introduce the set of $G$-valuations of $k(Z)$. Finally, we define the colored $\sigma$-polyhedral divisors in \S \ref{subsec13} and explain how to obtain any $G$-model of $Z$ from a finite collection of such polyhedral divisors.

The results of this paper are stated and proved in the second part. 
In \S \ref{subsec21} we explain how to obtain any simple $G$-model of $Z$ as the parabolic induction of an affine $L$-variety, where $L \subset G$ is a Levi subgroup. In particular, we obtain an effective construction of any simple $G$-model of $Z$. From this, we deduce several criteria to characterize the singularities of simple $G$-models of $Z$; see Theorem \ref{theorat} for a criterion for rationality of singularities and Theorems \ref{crit1} and \ref{crit2} for smoothness criteria. As mentioned at the end of \S \ref{subsec21}, our smoothness criteria are explicit thanks to the works of Pauer, Pasquier and Batyrev--Moreau. 
In \S \ref{subsec22} we prove the existence of the decoloration morphism for normal horospherical $G$-varieties of complexity one and give an explicit description of this morphism in terms of germs; see Proposition \ref{dec} for a precise statement. The decoloration morphism was introduced by Brion for spherical varieties in \cite[\S 3.3]{Bri91} and plays a key-role in our study of normal horospherical $G$-varieties of complexity one. Until the end of this introduction, we let $X$ be such a variety. In \S \ref{subsec23}, the heart of this paper, we parametrize the $G$-stable prime Weil divisors of $X$ and deduce from this a description of the class group of $X$ by generators and relations; see Theorem \ref{theodiv} and Corollary \ref{clgroup}. Then we obtain a criterion of factoriality for $X$; see Corollary \ref{cordiv}. Also, we relate the description of stable Cartier divisors obtained by Timashev in \cite{Tim00} to our description of stable Weil divisors; see Corollary \ref{corcartier}. In \S \ref{subsec24} we give an explicit representative of the canonical class of $X$; see Theorem \ref{theocan}. From this, we deduce criteria for $X$ to be $\QQ$-Gorenstein or log-terminal respectively; see Corollary \ref{Goren} and Theorem \ref{theolog}. Finally, Proposition \ref{res} provides an explicit resolution of singularities of $X$ which factors through the decoloration morphism defined in \S \ref{subsec22}.

\begin{Ack}
Both authors are grateful to Michel Brion, Ariyan Javanpeykar, Boris Pasquier, and especially to the referee for their reading and their valuable comments. The authors express their gratitude to the Institute of Mathematical Sciences (ICMAT) of Madrid for the hospitality it provided during the writing of this paper. The first-named author benefited from the support of the ERC Consolidator Grant NMST. He also thanks the Max Planck Institut für Mathematik Bonn for support. The second-named author benefited from the support of the DFG via the SFB/TR 45 ''Periods, Moduli Spaces and Arithmetic of Algebraic Varieties'' and from the support of the GDR ''Th\'eorie de Lie Alg\'ebrique et G\'eom\'etrique''.
\end{Ack}

\begin{notation} 
The base field $k$ is algebraically closed of characteristic zero. An integral separated scheme of finite type over $k$ is called a \emph{variety}. If $X$ is a variety, then $k[X]$ denotes the coordinate ring of $X$ and $k(X)$ denotes the field of rational functions of $X$. A point of $X$ is always assumed to be closed. We denote by $G$ a connected simply-connected reductive algebraic group (i.e., a direct product of a torus and a connected simply-connected semisimple group), by $B \subset G$ a Borel subgroup, by $U=R_u(B)$ the unipotent radical of $B$, and by $T \subset B$ a maximal (algebraic) torus. We denote by $\G_m$ the multiplicative group over $k$. A subgroup of $G$ is always a closed subgroup. If $H$ is such a subgroup, then 
$$N_{G}(H) = \{g\in G\ |\ gHg^{-1}\subset H\}$$
is the \emph{normalizer} of $H$ in $G$. For an algebraic group $K$, we denote by
$$ \chi(K) = \{ \text{algebraic group homomorphisms } \phi: K \to \G_m  \}$$
the character group of $K$. A variety on which $K$ acts (algebraically) is called a \emph{$K$-variety}. An algebra (over $k$) on which $K$ acts by algebra automorphisms is called a \emph{$K$-algebra}. If $X$ is a $K$-variety, then $k[X]$ and $k(X)$ are $K$-algebras; in particular, $k[X]$ and $k(X)$ are linear representations of $K$.
\end{notation}

\section{Preliminaries} \label{Sec1}

In this first part, we explain the combinatorial description of normal $G$-varieties of complexity one as given in \cite[\S 16]{Ti3} and specialized in the horospherical case. Let $C$ be a smooth projective curve, let $G/H$ be a horospherical $G$-homogeneous space, and let $Z=C \times G/H$; the group $G$ acts on $Z$ by translation on the second factor. The approach of Timashev consists in giving a classification of all the normal $G$-varieties which are $G$-birational to $Z$. 

Several times in this first part (and also later on) we mention the reference of a result (for instance in \cite{Ti3}) to justify an assertion or to introduce a notion but without recalling explicitly the result itself; the reason being that we did not wish to make this article too long and too technical. However, we tried to make it self-contained in the sense that whenever we believed that a notion or a result could be enlightening, then we recalled it explicitly.

\subsection{Horospherical homogeneous spaces} \label{pasquier}
In this section we enunciate a combinatorial description of the horospherical homogeneous spaces; see \cite[\S 2]{Pa} for details. 

Let $S$ be the set of simple roots of $G$ with respect to $(T,B)$. There exists a well-known one-to-one correspondence $I\mapsto P_{I}$ between the powerset of $S$ and the set of parabolic subgroups of $G$ containing $B$; see \cite[Th. 8.4.3]{Spr98}.
Let us assume that the closed subgroup $H \subset G$ contains the unipotent radical $U$ of $B$. Then $P=N_G(H)$ is a parabolic subgroup containing $B$. Therefore, there exists a unique subset $I \subset S$ such that $P=P_I$. The quotient algebraic group $K:=P/H$ is a torus and $M=\chi(K)$ identifies naturally with a sublattice of $\chi(T)$.

The next statement (\cite[Prop. 2.4]{Pa}) explains how the pair $(M,I)$ completely describes the horospherical homogeneous space $G/H$.

\begin{proposition} \label{des_ho}
The map $G/H\mapsto (M,I)$ is a bijection between
\begin{itemize}
\item[$\bullet$] the set of closed subgroups of $G$ containing $U$; and
\item[$\bullet$] the set of pairs $(M,I)$, where $M$ is a sublattice of $\chi(T)$ and $I$ is a subset of $S$ such that for every $\alpha \in I$ and every $m \in M$, we have $\langle m, \hat{\alpha}\rangle  = 0$ (here $\hat{\alpha}$ denotes the coroot of $\alpha$). 
\end{itemize}
\end{proposition}

\subsection{Models, colors, and valuations}  \label{subsec12}
From now on, $H$ is a closed subgroup of $G$ associated with a pair $(M,I)$ as in Proposition \ref{des_ho} and $P=P_I$ is the parabolic subgroup $N_G(H)$. Also, we recall that $Z=C \times G/H$.

\subsubsection{} \label{Models}
We first introduce the \emph{scheme of geometric localities} as in \cite[\S 12.2]{Ti3}. All varieties which are birational to $Z$ may be glued together into a scheme over $k$ that we denote by $\SC$. More precisely, the schematic points of $\SC$ are local rings corresponding to prime ideals of finitely generated subalgebras with quotient field $k(Z)$ and the spectra of those subalgebras define a base of the Zariski topology on $\SC$ (by identifying prime ideals with associated local rings). Furthermore, the abstract group $G$ acts on the set $\SC$ via its linear action on $k(Z)$. We denote by $\SG$
the maximal normal open subscheme on which the action of $G$ on $\SC$ is regular; see \cite[Prop. 12.2]{Ti3}. 
A \emph{$G$-model} of $Z$ is a $G$-stable dense open subset of $\SG$ which is separated and Noetherian (or equivalently separated and of finite type over $k$).

We now introduce the notions of chart and germ of a $G$-model $X$ of $Z$. 
A \emph{chart} (or \emph{affine chart} or \emph{$B$-chart}) of $X$ is an affine dense open subset of $X$ which is $B$-stable. 
A \emph{germ} (or a \emph{$G$-germ}) of $X$ is a non-empty $G$-stable irreducible closed subvariety $\Gamma \subsetneq X$. 
By \cite[Th. 1]{Sum74}, for every germ $\Gamma \subsetneq X$ there exists a chart $X_0\subset X$ such that $X_0\cap \Gamma \neq \emptyset$.
The $G$-model $X$ of $Z$ is called \emph{simple} if it has a chart intersecting all the germs. By \cite[\S 5, Lemma 2]{Tim00}, every simple $G$-model of $Z$ is quasi-projective. Moreover, $X$ is a finite union of simple $G$-models of $Z$.

\subsubsection{} \label{varrho}
We now introduce the notion of color. Let $K'$ be an algebraic group acting on a variety $X'$, then a \emph{$K'$-divisor} on $X'$ is an irreducible closed subvariety of $X'$ which is $K'$-stable and has codimension one. A \emph{color} of $G/H$ is a $B$-divisor on $G/H$.
Let us consider the natural map 
$$\pi: G/H\rightarrow G/P.$$ 
Then each color (of $G/H$) is of the form $D_{\alpha} = \pi^{-1}(E_{\alpha})$, where $E_{\alpha}$ is the Schubert variety of codimension one corresponding to the root $\alpha\in S\setminus I$. 

We represent colors as vectors of the lattice $N = {\Hom}_{\mathbb{Z}}(M,\mathbb{Z})$ as follows: For the natural action of $B$ on $k(G/H)$, the lattice $M=\chi(P/H)$ identifies with the lattice of $B$-weights of the $B$-algebra $k(G/H)$. For every (non-zero) $B$-eigenvector $f\in k(G/H)$ of weight $m\in M$, we put   
$$\langle m,\varrho(D_{\alpha})\rangle  = v_{D_{\alpha}}(f),$$
where $v_{D_{\alpha}}$ is the valuation associated with $D_{\alpha}$. The value $\varrho(D_{\alpha})$ does not depend on the choice of $f$ and coincides with the restriction of the coroot $\hat{\alpha}$ to the lattice $M$; see \cite[\S 2]{Pa}. 
Denoting by $\F_0$ the set of colors, we obtain a map $\varrho: \F_{0}\to N$. Let us note that $\varrho$ is not injective in general; for instance, if $H=P$ is a parabolic subgroup, then $N=\{0\}$ and thus $\varrho$ is constant.  

Let $X$ be a $G$-model of $Z$. A $B$-divisor of $X$ which is not $G$-stable is called a color of $X$. 
There is a one-to-one correspondence between the set of colors of $G/H$ and the set of colors of $X$ given as follows:
$X$ possesses a $G$-stable open subset of the form $C'\times G/H$, where $C'\subset C$ is a dense open subset. If $D$ is a color of $G/H$, then the closure of $C'\times D$ in $X$ is a color of $X$ and vice-versa. 
In the following, we will always denote the colors of $X$ and $G/H$ in the same way.

\subsubsection{}  \label{defE}
A \emph{$G$-valuation} of $k(Z)$ is a function $v:k(Z)^{\star} \to \QQ$ such that: 
\begin{itemize}
\item[$\bullet$] $v(a+b) \geq \min\{v(a), v(b)\}$, for all $a,b\in k(Z)^{\star}$ satisfying $a+b\in k(Z)^{\star}$;
\item[$\bullet$] $v$ is a group homomorphism from $(k(Z)^{\star},\times)$ to $(\QQ, +)$ whose image is a discrete subgroup of $(\QQ, +)$;
\item[$\bullet$] the subgroup $k^{\star}$ is contained in the kernel of $v$; and
\item[$\bullet$] $v(g\cdot a) = v(a)$ for every $g\in G$ and every $a\in k(Z)^{\star}$.
\end{itemize}
By \cite[Prop. 19.8]{Ti3}, every $G$-valuation of $k(Z)$ is proportional to a valuation $v_D$ for a $G$-divisor $D$ on a $G$-model of $Z$. Let us denote by $N_{\QQ} = \QQ\otimes_{\mathbb{Z}}N$ the $\QQ$-vector space associated with $N$. We follow \cite[Def. 16.1]{Ti3} and define the set $\E$ as the disjoint union of sets $\{z\}\times\E_z$, where $\E_z = N_{\QQ}\times\QQ_{\geq 0}$ and $z \in C$, modulo the equivalence relation $\sim$ defined by
\begin{equation} \label{mod}
(z,v,l)\sim (z', v', l') \text{ if and only if } z = z',\ v=v',\ l = l' \text{ or } v = v',\ l = l' = 0.
\end{equation}
Therefore, $\E$ is the disjoint union indexed by $C$ of copies of the upper half-space $N_{\QQ}\times\QQ_{\geq 0}\subset N_{\QQ}\times\QQ$ with boundaries $N_{\QQ}\times\{0\}$ identified as a common part.

There is a bijection between the set of $G$-valuations of $k(Z)$ and the set $\E$, which we now explain. Let $A_M$ denote the algebra generated by the $B$-eigenvectors of $k(Z)$. 
Since $M$ is a free abelian group, the exact sequence of abelian groups  
$$ 0 \rightarrow \left( k(Z)^B \right)^\star \rightarrow \left( k(Z)^{(B)} \right)^* \rightarrow M \rightarrow 0$$
splits. Let us fix once and for all a (non-canonical) splitting $M \to  \left( k(Z)^{(B)} \right)^*,\ m \mapsto \chi^m$. Then $A_M$ admits an $M$-grading given by
\begin{equation} \label{A_M}
A_M= \bigoplus_{m\in M}k(C)\chi^m.
\end{equation}
Let $u = [(z, v, l)] \in \E$. We define a valuation $w = w_{u}$ of $A_M$ as follows: 
$$w \left(\sum_{i\in I}f_{i}\chi^{m_{i}}\right) = \min_{i\in I} \left \{v(m_{i}) + l\cdot {\ord}_z(f_{i})\right \},$$
where $I$ is a finite set, the $m_i$ are pairwise distinct elements of $M$, and each $f_i$ belongs to $k(C)^{\star}$.  
By \cite[Cor. 19.13, Th. 20.3 and 21.10]{Ti3}, for every $u \in \E$, there exists a unique $G$-valuation of $k(Z)$ such that the restriction to $A_M$ is $w_u$. From now on, we will always identify $\E$ with the set of $G$-valuations of $k(Z)$ and $N_{\QQ}$ as a part of $\E$ via the (well-defined) map $v\mapsto [(\cdot,v,0)]$.

\subsection{Colored polyhedral divisors}  \label{subsec13}
Let $M_{\QQ} = \QQ\otimes_{\mathbb{Z}}M$, then the $\QQ$-vector spaces $M_\QQ$ and $N_\QQ$ are dual to each other; we denote the duality by   
$$M_{\QQ}\times N_{\QQ}\rightarrow \QQ, \ (m,v)\mapsto \langle m, v\rangle.$$
We recall that a strongly convex polyhedral cone in $N_\QQ$ is a cone generated by a finite number of vectors and which contains no line.

\subsubsection{}  \label{extraTim}

We now recall the notions of colored cone and colored hypercone since these notions will intervene several times in the following; see also \cite[\S\S 15.1, 16.3, and 16.4]{Ti3} for more details. 

\begin{definition} 
A \emph{colored cone} of $G/H$ is a pair $(\Cc,\F)$, where $\F \subset \F_0$ is a set of colors of $G/H$ such that $0 \notin \varrho(\F)$, and $\Cc \subset N_\QQ$ is a strongly convex polyhedral cone generated by $\varrho(\F)$ and finitely many other vectors.
\end{definition}

\begin{definition}
Let $C_0$ be a dense open subset of $C$. A \emph{hypercone} of $G/H$ is a union $\mathscr{C}=\bigcup_{z \in C_0} \{z\} \times \mathscr{C}_z$ of convex polyhedral cones $\mathscr{C}_z \subset N_\QQ \times \QQ_{\geq 0}$ such that
\begin{itemize}
\item for all but finitely many $z \in C_0$ we have $\mathscr{C}_z=\left(\mathscr{C} \cap N_\QQ \right)+\QQ_{\geq 0} \epsilon$ with $\epsilon=(0,\ldots,0,1) \in N_\QQ \times \QQ$; and
\item either (A) there exists $z \in C_0$ with $\mathscr{C}_z=\left(\mathscr{C} \cap N_\QQ \right)$, or \\
      \; \; \; (B) we have $\emptyset \neq \BB:=\sum_{z \in C_0} \BB_z \subset \left(\mathscr{C} \cap N_\QQ \right)$, where $\epsilon+\BB_z=\mathscr{C}_z \cap(\epsilon+N_\QQ)$. 
\end{itemize} 
We say that $\mathscr{C}$ is \emph{strongly convex} if all $\mathscr{C}_z$ are strongly convex and if $0 \notin \BB$.\\
A \emph{colored hypercone} of $G/H$ is a pair $(\mathscr{C},\F)$, where $\F \subset \F_0$ is a set of colors of $G/H$ such that $0 \notin \varrho(\F)$, and $\mathscr{C} \subset \E$ is a strongly convex hypercone whose each $\mathscr{C}_z$ is generated by $\varrho(\F) \times \{0\}$ and finitely many other vectors. \\
A \emph{hyperface} of a colored hypercone $(\mathscr{C},\F)$ is a colored hypercone $(\mathscr{C}',\F')$, where $\mathscr{C}'=\bigcup_{z \in C_0} \mathscr{C}'_z$ is a union of faces of $\mathscr{C}_z$, and $\F'=\F \cap \varrho^{-1}(\mathscr{C}')$.
\end{definition}

Also, we recall that there is a correspondence between the colored cones of a horospherical homogeneous space and its simple equivariant embeddings; see \cite[Th. 3.1]{Kn} for details.

\subsubsection{}   \label{kevin}
In this subsection we introduce the notion of colored polyhedral divisor. This notion is equivalent to the one of colored hypercone in $\E$ defined above, but is more suitable for our purposes. 

\begin{definition} \label{defpdiv1}
Let $\sigma\subset N_{\QQ}$ be a strongly convex polyhedral cone.  A \emph{$\sigma$-polyhedron} is a subset of $N_\QQ$ obtained as a Minkowski sum $Q + \sigma$, where $Q \subset N_{\QQ}$ is the convex hull of a non-empty finite subset. Let $C_0$ be a dense open subset of the curve $C$, let 
$$\D = \sum_{z\in C_0} \Delta_z \cdot [z]$$
be a formal sum over the points of $C_0$, where each $\Delta_z$ is a $\sigma$-polyhedron of $N_\QQ$ and $\Delta_z = \sigma$ for all but a finite number of $z \in C_0$, and let $\F \subset \F_0$ be a set of colors of $G/H$ such that 
\begin{itemize}
\item[$\bullet$] $0$ does not belong to $\varrho(\F)$; and
\item[$\bullet$] $\varrho(\F)\subset \sigma$.
\end{itemize}
We call such a pair $(\D,\F)$ a \emph{colored $\sigma$-polyhedral divisor} on $C_0$. 
If $\sigma$ and $\F$ are clear from the context, then we write $\D$ instead of $(\D,\F)$ and call $\D$ a colored polyhedral divisor on $C_0$.
We say that $\D$ is \emph{trivial} (on $C_0$) if $\Delta_z = \sigma$ for every $z \in C_0$. 
\end{definition}

It is crucial to keep in mind that a colored $\sigma$-polyhedral divisor is defined on a dense open subset $C_0 \subset C$ and not on $C$ itself.
In the following, $\sigma\subset N_{\QQ}$ always denotes a strongly convex polyhedral cone and $\F$ a set of colors satisfying the conditions of Definition \ref{defpdiv1}; in particular, $(\sigma,\F)$ is a colored cone of $G/H$. 

Let  
$$\sigma^{\vee} = \left\{ m\in M_{\QQ}\ \mid \ \forall v\in\sigma, \langle m, v\rangle \geq 0 \right\}$$
denote the dual polyhedral cone of $\sigma$, and let $\D$ be a colored $\sigma$-polyhedral divisor on $C_0 \subset C$. To each $m\in\sigma^{\vee}$, we associate a $\QQ$-divisor on $C_0$: 
\begin{equation} \label{D(m)}
\D(m) = \sum_{z\in C_{0}}\min_{v\in\Delta_z(0)}\langle m, v\rangle \cdot [z],
\end{equation}
where for every $z \in C_0$ we denote by $\Delta_z(0)$ the set of vertices of $\Delta_z$.

To a given $\D$ on $C_0$ we also associate the following $M$-graded normal $k$-algebra (see \cite[\S 3]{AH} for details): 
\begin{equation} \label{ACD}
A[C_0,\D]:=\bigoplus_{m\in\sigma^{\vee}\cap M}A_{m}\chi^{m},
\end{equation}
where
$$A_m = H^0\left( C_0,\mathcal{O}_{C_0}(\D(m))\right) :=  H^0\left(C_0,\mathcal{O}_{C_0}(\lfloor \D(m)\rfloor)\right)$$
and $\lfloor \D(m)\rfloor$ is the Weil divisor (with integer coefficients) on $C_0$ obtained by taking the integer part of each coefficient of $\D(m)$. 
The multiplication on $A[C_0,\D]$ is constructed from the maps 
$$\tau_{m,m'}: \ A_m\times A_{m'}\rightarrow A_{m+m'},\ (f_{1},f_{2})\mapsto f_{1}\cdot f_{2}.$$
Let us note that each map $\tau_{m,m'}$ is well-defined since for all $m,m'\in\sigma^{\vee}$:
$$\D(m) + \D(m')\leq \D(m+m').$$

We now introduce the localization of a colored $\sigma$-polyhedral divisor. We will use this notion in the proof of Theorem \ref{theodiv}.

\begin{definition} \label{localization}
With the notation above, let $w\in\sigma^{\vee}\cap M$ such that $A_{w}\neq \{0\}$ and $f\in A_w \setminus \{0\}$, and let $\F^w$ be the set of colors defined by the relation 
$$\F^w = \{D\in \F \ |\ \varrho(D)\in w^{\perp}\}.$$ 
The \emph{localization} of the colored $\sigma$-polyhedral divisor $(\D,\F)$ with respect to $f\chi^w$ is the colored ($\sigma\cap w^{\perp}$)-polyhedral divisor $(\D_f^w,\F^w)$ defined by 
$$\D_{f}^w = \sum_{z\in (C_0)_f^w}{\Face}(\Delta_z,w)\cdot [z],$$
where
$$(C_0)^w_f = C_0\backslash Z(f) \text{ with } Z(f) = {\Supp}({\div} f + \D(w)), \text{ and }$$
$${\Face}(\Delta_z, w) := \left\{v\in \Delta_z\ \middle| \ \langle w, v\rangle \leq \min_{v'\in\Delta_z}\langle w,v'\rangle \right\}.$$
\end{definition}

To ensure that the algebra $A[C_0,\D]$ is finitely generated over $k$ and has $\Frac A_M$ as field of fractions (we recall that $A_M$ denotes the algebra generated by the $B$-eigenvectors of $k(Z)$), we now introduce the notion of properness for colored polyhedral divisors following \cite[\S 2]{AH}.

\begin{definition} \label{properness}
Let $\D$ be a colored $\sigma$-polyhedral divisor on $C_0$. Then $\D$ is called \emph{proper} if either $C_0$ is affine or 
$C_0 = C$ is projective and satisfies the following conditions:
\begin{itemize}
\item[$\bullet$] $\deg \D := \sum_{z\in C}\Delta_z\subsetneq \sigma.$
\item[$\bullet$] If $\min_{v\in {\deg} \D}\langle m, v\rangle = 0$, then $r\D(m)$ is a principal divisor for some $r\in\mathbb{Z}_{>0}$.
\end{itemize}
\end{definition}

Let us note that, if $\D$ is a proper colored polyhedral divisor on $C_0$, then by \cite[Prop. 3.3]{AHS} we have the relation 
$$A[C_0,\D]_{f\chi^w} = A[(C_0)_f^w,\D_f^w].$$

The next remark makes the link between the notions introduced in this subsection and the notions introduced in \S \ref{extraTim}. 

\begin{remark}  \label{link}
Let $(\D,\F)$ be a colored $\sigma$-polyhedral divisor on a dense open subset $C_0\subset C$.
Let $\mathscr{C}(\D)$ be the subset of $\E$ defined as the disjoint union $\sqcup_{z \in C_0} \{z\}\times \mathscr{C}(\D)_z$ modulo the equivalence relation $\sim$ defined by \eqref{mod}, where $\mathscr{C}(\D)_z$ is the cone generated by $\sigma\times \{0\}$ and $\Delta_z\times\{1\}$. Then the pair $(\mathscr{C}(\D),\F)$ is the colored hypercone of $G/H$ associated with $(\D,\F)$. One may check that this gives a one-to-one correspondence between the set of colored polyhedral divisors defined on a dense open subset $C_0 \subset C$ and the set of colored hypercones of $\E$. 
Moreover, through this correspondence, the properness of colored polyhedral divisors corresponds to the \emph{admissibility} of colored hypercones; see \cite[Def. 16.12]{Ti3}.  
\end{remark}

\subsubsection{}  \label{simplemodels}
In this subsection, we explain how to construct a simple $G$-model of $Z$ (introduced in \S \ref{Models}) starting from a proper colored polyhedral divisor; see \cite[\S 13]{Ti3} for details.

Let $(\D,\F)$ be a proper colored $\sigma$-polyhedral divisor on a dense open subset $C_0\subset C$. 
Let us denote by $\mathscr{C}(\D)(1)$ the set of elements $[(z,v,l)]\in\mathscr{C}(\D)$ such that $(v,l)$ is the primitive vector of an extremal ray of $\mathscr{C}(\D)_z$, and let $Bx_0$ be the open $B$-orbit of $G/H$. Let us consider the subalgebra $A \subset k(Z)$ defined by:
$$A = (k(C)\otimes_{k}k[Bx_0])\cap\bigcap_{D\in \F}\mathcal{O}_{v_{D}}\cap\bigcap_{v\in\mathscr{C}(\D)(1)}\mathcal{O}_{v},$$
where for a discrete $G$-valuation $v$ of $k(Z)$, 
$$\mathcal{O}_{v} = \{f\in k(Z)^{\star}\ |\ v(f)\geq 0\}\cup\{0\}$$
is the corresponding local ring. By \cite[Th. 13.8, \S 16.4]{Ti3}, $A$ is a normal algebra of finite type over $k$. Moreover, by \cite[Cor. 13.9, \S 16.4]{Ti3}, the affine $B$-variety 
$$X_0(\D) := \Spec A$$ 
is an open subset of  $\SG$. Also, by \cite[Th. 12.6]{Ti3}, the subscheme 
$$X(\D):= G\cdot X_0(\D)$$ 
is a $G$-model of $Z$ and $X_0(\D)$ is a chart of $X(\D)$. In particular, $X_0(\D)$ is a dense open subset of $X(\D)$. 
Conversely, for every simple $G$-model $X$ of $Z$ there exists a colored polyhedral divisor $\D$ such that $X=X(\D)$; see \cite[Th. 16.19]{Ti3}.

\begin{lemma} \label{lemme-U-invariant}
Let $(\D,\F)$ be a proper colored $\sigma$-polyhedral divisor on a dense open subset $C_0\subset C$, and let $X_0(\D)\subset X(\D)$ be the corresponding chart. 
Then $k[X_0(\D)]^U$ identifies with $A[C_0,\D]$ as $K$-algebras, where $A[C_0,\D]$ is the algebra defined by \eqref{ACD} and $K$ is the torus introduced in \S \ref{pasquier}. 
\end{lemma}  

\begin{proof}
The subalgebra
$$k[X_0(\D)]^U \subset (k(C) \otimes_k k[Bx_0])^U=A_M \text{ (defined in \S \ref{defE})}$$
is generated by elements of the form $f \chi^m$, where $f \in k(C)^\star$ and $m \in M$, satisfying $v_D(f \chi^m) \geq 0$ and $v(f \chi^m) \geq 0$ for all $D \in \F$ and $v \in \mathscr{C}(\D)(1)$.  
One may check that these conditions are equivalent to $\div(f)+ \D(m) \geq 0$, which proves the lemma. 
\end{proof}  

\begin{remark}
By \cite[Th. 3.1]{Ti}, and with the notation of Definition \ref{localization}, we have the relation 
$$X(\D^w_f) = G\cdot (X_0(\D))_{f\chi^w},$$ 
where $(X_0(\D))_{f\chi^w}$ is the distinguished Zariski open subset of $f\chi^w \in k[X_0(\D)]$.  
\end{remark}

\subsubsection{}  \label{color_datum}
Let $X$ be the simple $G$-model of $Z$ associated with the proper colored $\sigma$-polyhedral divisor $(\D,\F)$ as explained in \S \ref{simplemodels}. 
By \cite[Th. 16.19 (2)]{Ti3}, there is a combinatorial description of the set of germs of $X$ (defined in \S \ref{Models}). 
Each germ $\Gamma$ is described by a pair $(\Cc_1, \F_1)$, called the \emph{colored datum} of $\Gamma$, where $\Cc_1$ is a hyperface of the hypercone $\mathscr{C}(\D)$ (see Remark \ref{link}) and $\F_1:= \{ D_\alpha \in \F \ |\ \varrho(D_\alpha) \in \Cc_1 \}.$ Geometrically, $\F_1$ is the subset of colors of $\F$ that contain the germ $\Gamma$.

\subsubsection{}
In this subsection we consider the general description of the non-necessarily simple $G$-models of $Z$.

A finite collection $\ES = \{(\D^i,\F^i)\}_{i\in J}$ of proper colored polyhedral divisors defined on dense open subsets of $C$ is a \emph{colored divisorial fan} if for all $i,j \in J$ there exists $l \in J$ such that $\mathscr{C}(\D^l)=\mathscr{C}(\D^i) \cap \mathscr{C}(\D^j)$ and $(\mathscr{C}(\D^l),\F^l)$ is a common hyperface of the colored hypercones $(\Cc(\D^i),\F^i)$ and $(\Cc(\D^j),\F^j)$.  

Let us denote 
$$|\ES| = \bigcup_{i\in J}\mathscr{C}(\D^i)\subset\E$$
the \emph{support} of the colored divisorial fan $\ES$. 

The next theorem gives a description of the $G$-models of $Z$; see \cite[Th. 16.19 (3)]{Ti3} and \cite[Cor. 12.14]{Ti3}.

\begin{theorem}
If $\ES = \{(\D^i,\F^i)\}_{i\in J}$ is a colored divisorial fan on $C$, then the union  
$X(\ES) := \bigcup_{i\in J}X(\D^i)$ in the scheme ${\rm Sch}_{G}(Z)$ (defined in \S \ref{Models}) is a $G$-model of $Z$, and every $G$-model of $Z$ is obtained in this way. 
Moreover, the $G$-model $X(\ES)$ of $Z$ is a complete variety if and only if $|\ES| = \E$, where $\E$ is the set of $G$-valuations (defined in \S \ref{defE}). 
\end{theorem}

\section{Main results}  \label{mainresults}

\subsection{Parabolic induction and smoothness criteria}   \label{subsec21}
In this section we prove Lemma \ref{lemme-Levi} which allows us to construct explicitly any simple $G$-model of $Z$ as the parabolic induction of an affine $L$-variety, where $L\subset G$ is a Levi subgroup. From this we deduce a criterion (Theorem \ref{theorat}) to determine whether the singularities of a simple $G$-model of $Z$ are rational. Moreover, we obtain smoothness criteria (Theorems \ref{crit1} and \ref{crit2}) for simple $G$-models of $Z$. Throughout this section we fix a proper colored $\sigma$-polyhedral divisor $(\D,\F)$ on a dense open subset $C_0 \subset C$.

\subsubsection{}  \label{Levi}
This subsection is inspired by \cite[\S 28]{Ti3}. 
Let us now introduce the notation that we will need to describe a simple $G$-model as the parabolic induction of an affine $L$-variety. 
As before, $T$ denotes a maximal torus of $G$ contained in $B$.
We consider the following set of simple roots
$$I': = \{\alpha\in S\setminus I\ |\ D_{\alpha}\in\F\}\cup I$$ 
and we denote by $P_\F$ the parabolic subgroup $P_{I'}$ containing $B$.
We choose a Levi subgroup $L \subset P_\F$ containing $T$ and a Borel subgroup $B_L$ of $L$ containing $T$ such that $I'$ is the set of simple roots of $L$ (with respect to $(T,B_L)$). To summarize, we have the following inclusions:
$$ T \subset B_L \subset L \subset P_\F \subset G.$$
We denote by $Z_L$ the horospherical $L$-variety $C\times L/H_L$, where $H_L = L\cap H$. By \cite[Cor. 15.6]{Ti3}, the homogeneous space $L/H_L$ is quasi-affine. Moreover, denoting $P_L = N_L(H_L)$, the torus $P_L/H_L$ identifies with $K=P/H$, and $M = \chi(P/H) = \chi(P_L/H_L)$ through this identification. Let $\F_L$ be the set of colors of $L/H_L$, then one may check that the image of $\F_L$ by the map $\varrho$ in the lattice $N$ is the same as the one of $\F$. In particular, $\F_L$ satisfies the conditions of Definition \ref{defpdiv1}. We denote by $(\D_L,\F_L)$ the colored $\sigma$-polyhedral divisor on $C_0$ (relatively to $Z_L$) defined by the formal sum $\D_L:=\D = \sum_{z\in C_0} \Delta_z \cdot [z]$, and by $X(\D_L) \subset {\Sch}_L(Z_L)$ the associated $L$-variety; see \S \ref{simplemodels}.    

The quotient morphism  $P_{\F}\to P_{\F}/R_u(P_\F) \cong L$ makes $X(\D_L)$ a $P_{\F}$-variety. As the quotient morphism $G \to G/P_\F$ is locally trivial (for the Zariski topology), we can form the \emph{twisted product} $G \times^{P_\F} X(\D_L)$. The latter is defined as the quotient $(G \times X(\D_L))/P_\F$, where $P_\F$ acts as follows:
$$ p.(g,x):=(gp^{-1},p.x), \ \text{ where } p \in P_\F, \ g \in G, \text{ and }\ x \in X(\D_L).$$
We recall that the twisted product $G \times^{P_\F} X(\D_L)$ is a locally trivial fiber bundle over $G/P_\F$ with fiber $X(\D_L)$; see \cite[\S I.5]{Jan87} for more details on twisted products.

\subsubsection{}
The next result is an adaptation of \cite[\S 3]{PV72} and \cite[\S 5]{Pau81} to the case of complexity one.

\begin{lemma} \label{lemme-Levi}
With the notation of \S \ref{Levi}, the $L$-variety $X(\D_L)$ is affine and $X(\D)$ is $G$-isomorphic to the twisted product $G\times^{P_{\F}}X(\D_L)$. 
Moreover, the $L$-algebra $k[X(\D_L)]$ identifies with the $L$-subalgebra
$$\mathcal{A}[C_0, \D_L]:=\bigoplus_{m\in\sigma^{\vee}\cap M}H^0\left(C_0,\mathcal{O}_{C_0}(\D(m))\right)\otimes_{k} V(m)\subset k(C)\otimes_{k} k[L/H_L],$$
where $\D(m)$ is the $\QQ$-divisor on $C_0$ defined by \eqref{D(m)} and $V(m)$ is the simple $L$-submodule of $k[L/H_L]$ of highest weight $m$ with respect to $(T,B_L)$. 
\end{lemma}

\begin{proof}
The $G$-isomorphism between $G\times^{P_{\F}}X(\D_L)$ and $X(\D)$ is a straightforward consequence of \cite[Prop. 14.4 and 20.13]{Ti3}. 
As $\F_L$ is the set of colors of $L/H_L$, the $B_L$-variety $X(\D_L)$ is the chart corresponding to $\D_L$ (see the remark before \cite[Cor. 13.10]{Ti3}), and thus $X(\D_L)$ is affine. 

As $X(\D_L)$ is horospherical, we have the following $L$-algebra decomposition (see \cite[Prop. 7.6]{Ti3}):
$$k[X(\D_L)]\simeq  \bigoplus_{m\in\sigma^{\vee}\cap M}k[X(\D_L)]^{(B_L)}_m\otimes_{k}V(m),$$ 
where $k[X(\D_L)]_m^{(B_L)}$ denotes the vector space generated by the $B_L$-eigenvectors of weight $m$ on which $L$ acts trivially. 
The vector space $k[X(\D_L)]_m^{(B_L)}$ identifies with the space of global sections of $\lfloor \D(m)\rfloor$. This proves the lemma.
\end{proof}

\begin{remark}
With the notation above, one may check that $X(\D_L)=X_0(\D_L)$.
\end{remark}

\begin{example}
Let us assume that $C_0=C$ is projective. Let $G=\SL_2$ and let $H$ be the unipotent radical of the subgroup of upper triangular matrices. Then $M=N=\ZZ$ and the set of colors $\F_0$ of $G/H$ is a singleton. Let $\sigma=\QQ_{\geq 0}$ and let $\D$ be a proper $\sigma$-polyhedral divisor on $C$ such that the Weil $\QQ$-divisor $\D(1)$ on $C$ is integral and very ample. Let us note that, in this situation, we have $G=L$ and thus $\D=\D_L$. Then by Lemma \ref{lemme-Levi}, the $G$-variety $X(\D,\F_0)$ is affine and $k[X(\D,\F_0)]$ identifies with
$$ \bigoplus_{d \geq 0} H^0(C,\mathcal{O}_C(d.\D(1))) \otimes_k V(d)$$
as a $G$-algebra, where $V(d)$ is the irreducible representation of $G$ of dimension $d+1$ obtained by linearizing the line bundle $\mathcal{O}_{\PP^1}(d)$. 
Then the line bundle $\mathscr{L}:=\mathcal{O}_C(\D(1)) \boxtimes \mathcal{O}_{\PP^1}(1)$ on $C \times \PP^1$ is very ample, and the $G$-variety $X(\D,\F_0)$ can be realized as the affine cone over the $G$-equivariant embedding of $C \times \PP^1$ in the projectivization of the space of global sections of $\mathscr{L}$. 
\end{example}

\subsubsection{}
The next results are criteria to characterize the singularities of a simple $G$-model $X(\D)$ of $Z$; see \cite[Prop. 5.1]{LS13} for the case of normal $T$-varieties and \cite[\S 6, Th. 7]{Tim00} for a criterion for rationality of singularities in the general setting of normal $G$-varieties of complexity one. 

We recall that a normal variety $X$ has \emph{rational singularities} if there exists a resolution of singularities $\phi: Y \to X$ such that the higher direct images of $\phi_*$ applied to $\mathcal{O}_Y$ vanish. This notion does not depend on the choice of the resolution of singularities.

\begin{theorem} \label{theorat}
Let $(\D,\F)$ be a proper colored $\sigma$-polyhedral divisor on $C_0$. The simple $G$-model $X(\D)$ of $Z$ has rational singularities if and only if one of the following assertions holds. 
\begin{enumerate}[(i)]
\item The curve $C_0$ is affine.
\item The curve $C_0$ is the projective line $\PP^{1}$ and ${\deg}\lfloor \D(m)\rfloor\geq -1$ for every $m\in\sigma^{\vee}\cap M$, where $\D(m)$ is the $\QQ$-divisor on $C_0$ defined by \eqref{D(m)}.
\end{enumerate}
\end{theorem}

\begin{proof}
In this proof, we identify $X(\D)$ with the parabolic induction of $X(\D_L)$ using Lemma \ref{lemme-Levi}. 
Let us denote by $B\bar{x}_0$ the open $B$-orbit of $G/P_{\F}$ and let 
$$q:X(\D) = G\times^{P_{\F}}X(\D_L)\rightarrow G/P_{\F}$$ 
be the ($G$-equivariant) projection. Then $q^{-1}(B\bar{x}_0)\simeq B\bar{x}_0\times X(\D_L)$ is a chart of $X(\D)$ intersecting all the germs of $X(\D)$. 
As $B\bar{x}_0$ is an affine space, it follows from \cite[Th. D5 (3)]{Ti3} that $X(\D)$ has rational singularities if and only if $X(\D_L)/\!/U_L\simeq \Spec A[C_0,\D]$ has rational singularities. We conclude by the rationality criterion for torus actions given in \cite[Prop. 5.1]{LS13}. 
\end{proof}

The next theorem gives a smoothness criterion for $X(\D)$ when $(\D,\F)$ is a proper colored polyhedral divisor on an affine curve $C_0$; see \cite[\S 4.2]{Bri91} for the spherical case, \cite[Prop. 5.1 and Th. 5.3]{LS13b} for the case of normal $T$-varieties of complexity one, and \cite[\S 3]{Mos1} for the case of embeddings of $\SL_2$ and $\PSL_2$. Our proof is inspired by a description of toroidal embeddings given in \cite[\S I\!I]{KKMS}. Roughly speaking, we prove that the smoothness of $X(\D)$, which we recall is a normal horospherical $G$-variety of complexity $1$, is tantamount to the smoothness of some normal horospherical $\G_m \times G$-varieties of complexity $0$.  

\begin{theorem}  \label{crit1} 
With the same notation as before, and assuming that $C_0$ is affine, the following statements are equivalent:
\begin{enumerate}[(i)]
\item The $G$-variety $X(\D)$ is smooth.
\item For every $z\in C_0$, the simple embeddings of the $\G_m\times G$-homogeneous space $\G_m\times G/H$ associated with the colored cones $(\mathscr{C}(\D)_z,\F)$ (see \S \ref{extraTim}) are smooth.
\end{enumerate}
\end{theorem}

\begin{proof}
Let us fix $z \in C_0$. As $C_0$ is smooth, by \cite[I\!V,17.11.4]{EGA}, there exist open subsets $U_z \subset C_0$ and $V_z \subset \AA^1$ containing $z$ and $0$ respectively, and there exists an \'etale morphism $\tau_z: U_z \to V_z$ such that $\tau_z(z)=0$. 
Let $\D'_L:=\sum_{x \in \AA^1} \Delta'_x \cdot [x]$ with $\Delta'_0=\Delta_z$ and $\Delta'_x=\sigma$ for all $x \neq 0$ be a polyhedral divisor on $\AA^1$. Let
${\D_L}_{|U_z}$ and ${\D'_L}_{|V_z}$ denote the polyhedral divisor obtained by restricting $\D_L$ and $\D'_L$ on $U_z$ and $V_z$ respectively.
One may check that the morphism $\tau_z$ induces an \'etale morphism of algebras $\phi_z: \mathcal{A}[V_z,{\D'_L}_{|V_z}] \to \mathcal{A}[U_z,{\D_L}_{|U_z}]$. 
By Lemma \ref{lemme-Levi}, the morphism $\phi_z$ in turn induces an \'etale morphism $\delta_z: X({\D_L}_{|U_z}) \to X({\D'_L}_{|V_z})$. Indeed, identifying $X({\D_L}_{|U_z})$ with $X({\D'_L}_{|V_z}) \times_{V_z} U_z$ via \'etale base change, the morphism $\delta_z$ is simply given by the first projection.

Let us denote by $\Xi_L$ the colored cone $(\mathscr{C}(\D_L)_z,\F_L)$ and by $X_{\Xi_L}$ the corresponding embedding of $\G_m \times L/H_L$; then $X_{\Xi_L}$ is $L$-isomorphic to $X(\D'_L)$. Let $\gamma_z$ be the morphism which makes the diagram
$$\xymatrix{
    X({\D_L}_{|U_z}) \ar[d]_{\delta_z} \ar[r]^{\gamma_z} & X_{\Xi_L}  \ar[d]^{\cong}\\
    X({\D'_L}_{|V_z}) \ar@{^{(}->}[r]  & X(\D'_L)  
  }$$
commute, where the bottom arrow is an open embedding. In particular, $\gamma_z$ is an \'etale morphism.

Let us prove $(ii) \Rightarrow (i)$. Denote by $\Xi$ the colored cone $(\mathscr{C}(\D)_z,\F)$, and let $X_\Xi$ be the corresponding embedding of $\G_m \times G/H$. Then we have a $G'$-isomorphism $X_\Xi \cong G' \times^{P'_\F} X_{\Xi_L}$ (see \cite[Th. 28.2]{Ti3}), where $G'=\G_m \times G$ and $P'_\F \subset G'$ is the parabolic subgroup constructed as in \S \ref{Levi}. By assumption, $X_\Xi$ is smooth and thus so is $X_{\Xi_L}$. This implies that $X({\D_L}_{|U_z})$ is also smooth. Since $C_0$ is affine, $(X({\D_L}_{|U_z}))_{z \in C_0}$ is an open covering of $X(\D_L)$. Therefore $X(\D_L)$ is smooth and thus, by parabolic induction, so is $X(\D)$.
      
Let us prove $(i) \Rightarrow (ii)$. If $X(\D)$ is smooth, then so is $X(\D_L)$ by parabolic induction. Hence, by the diagram above, there exists an open subset $V \subset \AA^1$ containing $0$ such that $X({\D'_L}_{|V})$ is smooth. Denoting $L'=\G_m \times L$ and identifying $X({\D'_L}_{|V})$ with an open subset of $X_{\Xi_L}$, we have $L' \, \cdot \, X({\D'_L}_{|V}) = X_{\Xi_L}$. This implies that $X_{\Xi_L}$ is smooth, and thus so is $X_\Xi$.       
\end{proof}

We say that two proper colored polyhedral divisors $\D$ and $\D'$ on $C_0$ are \emph{equivalent} if $A[C_0,\D]$ and $A[C_0,\D']$ (defined by \eqref{ACD}) are isomorphic as $M$-graded algebras; see \cite[\S 8]{AH} and \cite[Prop. 4.5]{La} for a combinatorial description of the equivalence between two such polyhedral divisors. If $\D$ and $\D'$ are equivalent this does not imply a priori that $X(\D)$ and $X(\D')$ are isomorphic.

The next theorem gives a smoothness criterion for $X(\D)$ when $(\D,\F)$ is a proper colored polyhedral divisor on a projective curve $C_0=C$. 

\begin{theorem}  \label{crit2} 
With the same notation as before, and assuming that $C_0$ is projective (i.e., $C_0=C$), the following statements are equivalent:
\begin{enumerate}[(i)]
\item The $G$-variety $X(\D)$ is smooth.
\item The curve $C$ is $\PP^1$, the polyhedral divisor $\D$ is equivalent to a proper colored polyhedral divisor $\D^{0,\infty}=\sum_{z \in \PP^1} \Delta_z \cdot [z]$ with $\Delta_z=\sigma$ except when $z=0$ or $\infty$, and the simple embedding of the $\G_m\times G$-homogeneous space $\G_m\times G/H$ associated with the colored cones $(\Cc,\F)$ (see \S  \ref{extraTim}) is smooth, where $\Cc$ is the cone generated by $(\sigma \times \{0\}) \cup (\Delta_0 \times \{1\}) \cup (\Delta_{\infty} \times \{-1\})$. 
\end{enumerate}
\end{theorem}

\begin{proof}
Let us prove $(i) \Rightarrow (ii)$. Suppose that $X=X(\D)$ is smooth. Then by Lemma \ref{lemme-Levi} we can assume that $X$ is affine and identify $k[X]$ with $$\mathcal{A}[C, \D]=\bigoplus_{m\in\sigma^{\vee}\cap M}H^0\left(C,\mathcal{O}_{C}(\D(m))\right)\otimes_{k} V(m).$$ 
Moreover, we can assume that $\sigma^{\vee}$ is strongly convex; otherwise, there is a non-trivial torus $D$ and a $G$-variety $X'$ such that $X \cong D \times X'$ and then we replace $X$ by $X'$. 
By Luna's slice theorem (see \cite[\S I\!I\!I, Cor. 2]{Luna}), there is a $G$-isomorphism $X \cong G \times^F V$, where $F \subset G$ is a reductive closed subgroup and $V$ is a $F$-module. It follows from the proof of Luna's slice theorem that $F$ is in fact the stabilizer of a point of $X$ and thus $F$ is a horospherical subgroup. Since $F$ is reductive and contains a maximal unipotent subgroup of $G$, it contains the semisimple part of $G$, and thus $G/F$ is a torus. Now the surjective map $G \times^F V \to G/F$ induces an inclusion $k[G/F] \subset k[G \times^F V]=k[X]$. Since $\sigma^{\vee}$ is strongly convex and $\D$ is proper,
we have $k[G/F]=k$ and so $G=F$. Therefore we obtain that $X$ is $G$-isomorphic to the $G$-module $V$. Let us identify $X$ with $V$ and let us denote by $\gamma: G \to GL(X)$ the corresponding homomorphism. If $T \subset B$ is a maximal torus, then there exists a maximal torus $\mathbb{T} \subset GL(X)$ containing $\gamma(T)$ and normalizing $\gamma(U)$, where $U$ is the unipotent radical of $B$. Therefore $X/\!/U$ is a toric variety for the action of $\mathbb{T}$. 

It follows from \cite[\S 11]{AH} that $C=\PP^1$ and $\D$ is equivalent to a polyhedral divisor $\D^{0,\infty}$ supported by $0$ and $\infty$. 
We fix an isomorphism of $M$-graded algebras $\phi: A[C,\D] \to A[C,\D^{0,\infty}]$. Then there exists a group homomorphism 
$$(M,+) \to (k(C)^{\star}, \times),\ \; m \mapsto f_m$$ 
and a linear automorphism $\delta: M \to M$ preserving $\sigma^\vee \cap M$ such that for every $m \in \sigma^\vee \cap M$, we have $\phi(f\chi^m)=f f_{m}^{-1} \chi^{\delta(m)}$. Let us denote by the same letter $\delta: N \to N$ the dual linear map of $\delta$. We define a new $\sigma$-polyhedral divisor on $C$ by $$\delta_\star (\D^{0,\infty})=\sum_{y \in C} \delta(\Delta_y^{0,\infty}) \cdot [y],$$ 
where $\Delta_y^{0,\infty}$ is the coefficient of $\D^{0,\infty}$ at the point $y \in C$. Then $\delta_\star (\D^{0,\infty})(m)=\D^{0,\infty}(\delta(m))$ for every $m \in \sigma^\vee \cap M$, and we have a $G$-isomorphism 
$$\mathcal{A}[C,\D] \to \mathcal{A}[C,\delta_\star(\D^{0,\infty})],\  f \otimes v \mapsto ff_{m}^{-1} \otimes v$$
with $f \in k(C)^\star$ and $v \in V(m)$. Therefore $X=\Spec \mathcal{A}[C,\D]$ identifies with the horospherical variety associated with the colored cone $((\delta \times Id)(\Cc),\F)$. Hence, by the smoothness criterion of \cite[\S 5]{BM13}, this is equivalent to the smoothness of the horospherical variety associated with the colored cone $(\Cc,\F)$. 

The converse implication $(ii) \Rightarrow (i)$ is proved by using the same kind of arguments and is left to the reader.    
\end{proof}

The smoothness criteria of Theorems \ref{crit1} and \ref{crit2} can be made explicit by applying the smoothness criterion in the horospherical embedding case given by the following theorem; see \cite[\S 3.5]{Pau83}, \cite[\S I\!I]{Pas}, and \cite[\S 5]{BM13}. 
\begin{theorem} \label{explicit}
Let $X$ be a $G$-equivariant embedding of the horospherical homogeneous space $G/H$ associated with a colored cone $(\Cc,\F)$ (see \S \ref{extraTim}). 
Then $X$ is smooth if and only if the following conditions are satisfied.
\begin{enumerate}[(i)]
\item The elements of $\F$ have pairwise distinct images through the map $\varrho$ defined in \S \ref{varrho}. 
\item The cone $\Cc$ is generated by a subset of a basis of $N$ containing $\varrho(\F)$.
\item We have the equality 
$$|W_I| \cdot \prod_{\alpha \in I_{\F}} a_\alpha=|W_{I \cup I_{\F}}|,$$
where $I$ is the subset of simple roots of $G$ defined in \S \ref{pasquier}, $I_\F=\{\alpha \in S \setminus I \ |\ D_\alpha \in \F\}$, $W_I$ is the subgroup of the Weyl group $W=N_G(T)/T$ generated by the simple reflexions $s_\alpha$ $(\alpha \in I)$, $a_\alpha=\left\langle \sum_{\beta \in \Phi^+ \setminus \Phi_I} \beta, \hat \alpha \right\rangle$, $\Phi^+$ is the set of positive roots with respect to $(T,B)$, and $\Phi_I$ is the set of roots that are sums of elements of $I$.
\end{enumerate}
\end{theorem}

\subsection{Decoloration morphism}  \label{subsec22}

In this section, we prove the existence of the decoloration morphism for the normal horospherical $G$-varieties of complexity one; see \cite[\S 3.3]{Bri91} for the spherical case. This will be used several times in a crucial way to prove our statements in the following, and also to construct an explicit resolution of singularities of $X(\Sigma)$ in Proposition \ref{res}.

\begin{definition}
Let $\ES = \{(\D^i,\F^i)\}_{i\in J}$ be a colored divisorial fan, then the \emph{decoloration} of $\ES$ is the colored divisorial fan 
$$\ES_{\dc}:= \{(\D^i,\emptyset)\}_{i\in J}.$$ 
A $G$-model $X(\ES)$ of $Z$ such that $\ES=\ES_{\dc}$ is called \emph{quasi-toroidal}.
As a consequence of the description of the germs of a $G$-model of $Z$, this notion does not depend on the choice of $\ES'$. 
Likewise, if $(\Cc_0, \F_0)$ is the colored datum (see \S \ref{color_datum}) of some germ $\Gamma' \subset X(\ES)$, then we say that the germ $\Gamma' \subset X(\ES_{\dc})$ corresponding to the colored datum $(\Cc_0, \emptyset)$ is the \emph{decoloration} of $\Gamma'$.
\end{definition}

Let $\ES = \{(\D^i,\F^i)\}_{i\in J}$ be a colored divisorial fan, then the collection of $k$-schemes $Y(\D^i) = \Spec A[C_0^i,\D^i]$, equipped with their $K$-action, glue together to give a normal $K$-variety of complexity one that we denote by $Y(\ES)$; see \cite[Th. 5.3 and Remark 7.4 (ii)]{AHS}. 

Before stating the next result, let us introduce some notation. 
Let $\Gamma \subset X$ be a germ of a $G$-model of $Z$; we say that $\Gamma$ is a \emph{geometric realization} of $\mathcal{O}_{X,\Gamma}$ in the variety $X$. The \emph{support} of $\Gamma$, denoted by ${\Supp}(\Gamma)$, is the set of $G$-valuations $v$ of $k(Z)$ such that $\mathcal{O}_v$ dominates the local ring $\mathcal{O}_{X,\Gamma}$.

\begin{proposition} \label{dec} 
Let $\ES = \{(\D^i,\F^i)\}_{i\in J}$ be a colored divisorial fan, and let $X_0^i$ and $X_{\dc}^i$ be the charts corresponding to $(\D^i,\F^i)$ and $(\D^i,\emptyset)$ respectively. The inclusions $k[X_0^i]\subset k[X_{\dc}^i]$ induce a proper birational $G$-morphism 
$$\pi_{\dc}:X(\ES_{\dc})\rightarrow X(\ES).$$ 
Moreover, for every germ $\Gamma\subset X$, the subset $\pi_{\dc}^{-1}(\Gamma)$ is the germ obtained by decoloring the colored datum of $\Gamma$. 
Also, there exists a $G$-isomorphism between $X(\ES_{\dc})$ and the twisted product $G/H\times^{K}Y(\ES)$, where $Y(\ES)$ is the $K$-variety defined above and $K=P/H$ acts on $G/H$ as follows: for every $g\in G$, for every $pH \in K$, we have $pH\cdot gH = gp^{-1}H$. 
\end{proposition}

\begin{proof}
Let us fix an index $i\in J$. The inclusion $k[X_0^i]\subset k[X_{\dc}^i]$ induces a birational $B$-morphism $\iota_0:X^i_{\dc} \to X^i_0$. The latter induces a $G$-equivariant birational map $\iota: X(\D^i,\emptyset) \dashedrightarrow X(\D^i,\F^i)$. By \cite[Prop. 12.12]{Ti3}, the morphism $\iota$ extends to a $G$-morphism 
$$\pi_{\dc}^i: X(\D^i,\emptyset)\rightarrow X(\D^i,\F^i)$$
if and only if for every germ $\Gamma \subset X(\D^i,\emptyset)$, there exists a (necessarily unique) germ $\Gamma'\subset X(\D^i, \F^i)$ such that the local ring $\mathcal{O}_{X(\D^i,\emptyset), \Gamma}$ dominates $\mathcal{O}_{X(\D^i,\F^i), \Gamma'}$. 
Let us consider the colored datum $(\Cc^i_0,\emptyset)$ of the germ $\Gamma\subset X(\D^i,\emptyset)$, and denote
$$\F_0^i = \{D\in\F^i\ |\ \varrho(D)\in \Cc_0^i\}.$$
Then $(\Cc_0^i,\F_0^i)$ is the colored datum of a germ $\Gamma'\subset X(\D^i,\F^i)$. 
By \cite[Prop. 14.1 (2)]{Ti3} or \cite[\S 3.8]{Kno93}, the support of a germ associated to a color datum $(\Cc',\F')$ depends only on $\Cc'$, and thus ${\Supp}(\Gamma) = {\Supp}(\Gamma')$. 
This implies the equality $\mathcal{O}_{X(\D^i,\emptyset), \Gamma} = \mathcal{O}_{X(\D^i,\F^i), \Gamma'}$; see the proof of \cite[Prop. 14.1(1)]{Ti3}. In particular, $\mathcal{O}_{X(\D^i,\emptyset), \Gamma}$ dominates $\mathcal{O}_{X(\D^i,\F^i), \Gamma'}$ and thus $\iota$ extends to a $G$-morphism. 

Let now $\Gamma' \subset X(\D^i,\F^i)$ be an arbitrary germ. Since the induced map 
$(\pi_{\dc}^{i})_{*}$ 
is the identity on $\mathscr{E}$ (see the remark before \cite[Th. 12.13]{Ti3}), the subset $(\pi_{\dc}^i)^{-1}(\Gamma')$ is irreducible and 
coincides with the decoloration of $\Gamma'$.

The properness of $\pi_{\dc}^i$ follows from the properness criterion given by \cite[Th. 12.13]{Ti3}. 
The existence and the properties of the morphism 
$$\pi_{\dc}:X(\ES_{\dc})\rightarrow X(\ES)$$ 
are then obtained by gluing.

For the last claim, it suffices to show that $X(\D,\emptyset)$ is $G$-isomorphic to $G/H\times^{K}Y(\D)$. The latter identifies with the $G$-variety $G\times^{P}Y(\D)$ obtained by parabolic induction. Now it follows from \cite[Prop. 14.4 and 20.13]{Ti3} that $X(\D,\emptyset)$ and the $G\times^{P}Y(\D)$ have the same combinatorics data, whence the result.
\end{proof}

\begin{example} 
We consider the natural action of $G = \SL_{3}$ on $\AS^{3} = \mathbb{A}^{3}\setminus \{(0,0,0)\}$. Let $H$ be the isotropy subgroup of the point $(1,0,0)$ for this action. Then $H$ is a horospherical subgroup of $G$ and $\AS^{3} \cong G/H$. Also, the torus $K=N_G(H)/H \cong \G_m$ acts diagonally on $\AS^{3}$ and the fibration $G/H = \AS^{3}\rightarrow G/P = \PP^{2}$ is simply the quotient morphism for the $\G_m$-action.
   
Let us consider the colored $\sigma$-polyhedral divisor on $\AA^1=\Spec k[t]$ defined by $\F = \emptyset$ and $\D = [\frac{1}{2}, +\infty[\cdot [0]$, where 
$N_{\QQ} = \QQ$, and $\sigma  = \QQ_{\geq 0}$. 
The $k$-algebra 
$$A[\AA^{1},\D] = \bigoplus_{m\geq 0}k[t]t^{-\lfloor \frac{1}{2}m\rfloor}\chi^{m}$$
is generated by the homogeneous elements $t, \chi^{1},$ and $\frac{1}{t}\chi^{2}$.
Therefore, $Y(\D)=\Spec A[\AA^{1},\D]$ can be identified with the affine surface $V(xz-y^{2})\subset \AA^{3}$ equipped with the $\G_{m}$-action defined by $\lambda\cdot (x,y,z) = (x, \lambda^{-1} y, \lambda^{-2} z)$ with $\lambda\in \G_{m}$.
 
Denoting by $(x_1,x_2,x_3)$ a system of coordinates of $\mathbb{A}^{3}$, the twisted action of  
$\G_{m}$ on the product $G/H\times Y(\D)$ is given by  
$$\lambda \cdot (x_{1}, x_{2}, x_{3}, x, y, z) = (\lambda^{-1} x_{1}, \lambda^{-1} x_{2}, \lambda^{-1} x_{3}, x, \lambda^{-1} y, \lambda^{-2} z).$$ 
By Proposition \ref{dec}, the $G$-variety $X:=X(\D,\emptyset)$ identifies with the quotient $G/H\times^{K}Y(\D)$. 
Hence, $X$ is the hypersurface $xy-z^{2} = 0$ in the complement of  
$$\{[0:0:0:x:y:z]\,|\, [x:y:z]\in \PP(0,-1,-2)\}$$
in the weighted projective space $X':=\PP(-1, -1, -1, 0, -1,-2)$. 
Let $B \subset G$ be the Borel subgroup of upper triangular matrices. To determine a chart of $X$, it suffices to determine the inverse image of the open $B$-orbit in $\PP^{2} = G/P$ through the projection $q:X=G/H\times^{K}Y(\D) \to G/P$. The open orbit $B\bar{x}_{0}\subset \PP^{2}$ is precisely  
$\PP^{2}\setminus\{x_{3} = 0\}\simeq \AA^{2}$. Thus the chart $X_0(\D)=q^{-1}(B\bar{x}_{0})$ is the hypersurface $xy-z^{2} = 0$ in $X'\setminus \{x_{3} = 0\}$ which is isomorphic to $\AA^{2}\times V(xy-z^{2})$. 
\end{example}

\subsection{Parametrization of the stable prime divisors} \label{subsec23}  

In this section, we start by describing in Theorem \ref{theodiv} the germs of codimension one of a normal horospherical $G$-variety of complexity one $X$. From this, we deduce a description of the class group of $X$ by generators and relations; see Corollary \ref{clgroup}. Next, we obtain a factoriality criterion for $X$; see Corollary \ref{cordiv}. Finally, in \S \ref{Cartier}, we relate the description of stable Cartier divisors obtained by Timashev in \cite{Tim00} to our description of stable Weil divisors.     

\subsubsection{}  \label{RayVert}
To state our results we need first to introduce the set of vertices and the set of extremal rays of a colored polyhedral divisor.
Let $(\D,\F)$ be an element of a colored divisorial fan $\ES$ with $\D = \sum_{z\in C_0}\Delta_z\cdot [z]$.
The \emph{set of vertices of $\D$}, denoted by $\Vert(\D)$, consists in pairs $(z,v)$ where $z\in C_0$ and $v\in\Delta_z(0)$ is a vertex of $\Delta_z$.
If $\ES = \{(\D^i,\F^i)\}_{i\in J}$, then we put 
$$\Vert(\ES):= \bigcup_{i\in J}\Vert(\D^i)\subset C\times N_{\QQ}.$$
The \emph{set of extremal rays} of $\D$, denoted by $\Ray(\D)$ or $\Ray(\D,\F)$, consists in extremal rays $\rho\subset \sigma$ such that $\rho\cap \varrho(\F) = \emptyset$, and satisfying $\rho\cap {\deg}\D = \emptyset$ when $C_0=C$.  
To simplify the notation, we denote by the same letter an extremal ray of a polyhedral cone of $N_{\QQ}$ and its primitive vector with respect to the lattice $N$. 
We also denote 
$$\Ray(\ES):= \bigcup_{i\in J}\Ray(\D^i)\subset N_{\QQ},$$
where we recall that $N_\QQ$ naturally identifies with a subset of $\E$; see \S \ref{defE}.
Finally, we denote by $C_{\ES}$ the union of open subsets $\bigcup_{i\in J}C_0^i\subset C$, where $C_0^i$ is the curve on which $\D^i\in \ES$ is defined.

\subsubsection{}
In the next theorem we parametrize the set of $G$-divisors of a $G$-model $X(\ES)$ of $Z$ by the set $\Vert(\ES)\coprod\Ray(\ES)$. 
This description is a natural generalization of the case of normal $T$-varieties specialized to the case of $T$-actions of complexity one; see \cite[Th. 4.22]{FZ} and \cite[Prop. 3.13]{PS}. 

\begin{theorem} \label{theodiv}
Let ${\Div}(\ES)$ denote the set of $G$-divisors of $X(\ES)$. With the notation above, the map
$$ \Vert(\ES)\coprod\Ray(\ES)\rightarrow {\Div}(\ES),\ \ (z,v)\mapsto D_{(z,v)} ,\  \rho\mapsto D_{\rho}$$
which to the vertex $(z, v)$ associates the germ $D_{(z,v)}$ of $X(\ES)$ defined by the colored datum $([(z, \QQ_{\geq 0}(v,1))],\emptyset)$ resp. to the ray $\rho$ associates the germ $D_\rho$ of $X(\ES)$ defined by the colored datum $(\rho,\emptyset)=([(\cdot,\rho,0)],\emptyset)$, is a bijection (see \S \ref{color_datum} for the definition of colored datum).
\end{theorem}

\begin{proof}
Without loss of generality, we can suppose that $X(\ES) = X(\D)$ is given by a proper colored $\sigma$-polyhedral divisor $(\D,\F)$ on a dense open subset $C_0\subset C$, where $\D =\sum_{z\in C_0}\Delta_z\cdot [z]$. 
By \cite[Th. 16.19 (2)]{Ti3}, the maximal germs of $X(\D)$ have color data of the form $(\rho,\F_{1})$ and $([(z, \QQ_{\geq 0}(v,1))],\emptyset)$ where 
$\rho\subset \sigma$ is an extremal ray, $\F_{1} = \{D_{\alpha}\in \F \ |\ \varrho(D_{\alpha})\in \rho\}$, 
$z\in C_0$, and $v\in\Delta_z(0)$. 
We are going to examine in each case the maximal germs corresponding to $G$-divisors. We proceed in three steps.

\emph{Step 1:} We consider the case where $X(\D)$ is quasi-toroidal, i.e., $\F=\emptyset$. Then, by Proposition \ref{dec}, $X(\D)$ identifies with $G/H\times^{K}Y(\D)$ as a $G$-variety. 
By \cite[I.5.21(1)]{Jan87}, there is a natural bijection between the set of $K$-divisors on $Y(\D)$ and the set of $G$-divisors on $G/H\times^{K}Y(\D)$. 
By \cite[Prop. 3.13]{PS}, the set of $K$-divisors on $Y(\D)$ is parametrized by the set $\Vert(\D)\coprod\Ray(\D)$:
$$(z,v)\mapsto \Gamma_{(z,v)},\ \rho\mapsto \Gamma_{\rho}.$$
The $G$-valuation $v_{D_{\rho}}$ with $D_{\rho} = G/H\times^{K}\Gamma_{\rho}$ resp. $v_{D_{(z,v)}}$ with $D_{(z,v)} = G/H\times^{K}\Gamma_{(z,v)}$, is represented by $[(\cdot,\rho,0)]\in \E$ resp. by $[(z,\mu(v)(v, 1))] \in\E$, where $\mu(v):=\inf\{d\in \ZZ_{>0}\ |\ dv \in \NN \}$. This parametrization of the $G$-divisors on $X(\D)$ by the set $\Vert(\D)\coprod\Ray(\D)$ is the one requested.

\emph{Step 2:} We now consider an arbitrary simple $G$-model $X(\D,\F)$ of $Z$, and let 
$$\pi_{\dc}: X(\D,\emptyset)\rightarrow X(\D,\F)$$
be the decoloration morphism. By Proposition \ref{dec}, any $G$-divisor on $X(\D,\F)$ is the image of a (unique) $G$-divisor on $X(\D,\emptyset)$. 
Let us consider the $G$-divisor $D'_{\rho}$ on $X(\D,\emptyset)$ corresponding to the color datum $([(\cdot,\rho, 0)],\emptyset)$, where $\rho$ belongs to $\Ray(\D,\emptyset)$. Then $D_{\rho}=\pi_{\dc}(D'_{\rho})$ is the germ of $X(\D,\F)$ corresponding to the color datum $([(\cdot,\rho, 0)], \F_1)$. In this second step, we want to prove that $D'_{\rho}$ is contracted by $\pi_{\dc}$ if and only if $\F_1 \neq \emptyset$. 
   
Let $\rho^{\star} = \rho^{\perp}\cap\sigma^{\vee}\subset\sigma^{\vee}$ denote the dual face of $\rho$. By properness of $\D$ (see \cite[\S 2]{AH}) there exists a homogeneous element $f\chi^m\in A[C_0,\D]$ of degree $m$ belonging to the relative interior of $\rho^{\star}$ such that
$$\{z\in C_0\ |\ \Delta_z\neq \sigma\}\subset Z(f) = {\Supp}({\div} f + \D(m))\subset C_0.$$
The localization $\D_f^m$ of $\D$ with respect to $f\chi^m$ is the colored $\rho$-polyhedral divisor trivial on the curve $(C_0)_f^m = C_0\backslash Z(f)$ with set of colors $\F_1$; see Definition \ref{localization}. We denote by $X'$ the simple embedding of $G/H$ associated with the colored cone $(\rho,\F_1)$ (see \S  \ref{extraTim}).
By computing in an appropriate chart, one checks that the product $(C_0)_f^m\times X'$ identifies with $X(\D_f^m,\F_1)$ as a $G$-variety. We denote by $Gx$ the (unique) closed orbit of $X'$. Then $D''_{\rho}:= (C_0)_f^m\times Gx$ is a maximal germ (for the inclusion) of $X(\D_f^m,\F_1)$ and corresponds thus to the colored datum $(\rho,\F_1)$; see \cite[Th. 16.19]{Ti3} for details. The germ $D''_{\rho}$ is a geometric realization of $\mathcal{O}_{X(\D,\F),D_{\rho}}$ in the open subset $X(\D_f^m,\F_1)\subset X(\D,\F)$, that is, $D_{\rho}$ is the closure of $D''_{\rho}$ in $X(\D,\F)$. Therefore, $D_{\rho}$ is a divisor on $X(\D,\F)$ if and only if $G x$ is a divisor on $X'$ if and only if $\F_1 = \emptyset$; see \cite[Lemma 2.4]{Kn} for the last equivalence.

\emph{Step 3:} It remains to study the germs of the form $D_{(z,v)}\subsetneq X(\D,\F)$. Let us fix a vertex $(z,v)\in\Vert(\D)$, and let us write $\sigma^{\vee}$ as a union $$\sigma^{\vee} = \bigcup_{w\in\Delta_z(0)}\sigma^{\vee}_w,\text{  where  }  
\sigma^{\vee}_w = \left\{m\in\sigma^{\vee} \ \middle| \ \langle m, w\rangle  =\min_{v'\in\Delta_z(0)}\langle m, v'\rangle \right\}.$$
The cones $\sigma^{\vee}_w$ are pairwise distinct, generate a quasi-fan of $\sigma^{\vee}$ (see \cite[\S 1]{AH}) and are all of full dimension in $M_{\QQ}$. Let $m\in\sigma^{\vee}_{v}\cap M$ be a lattice vector in the relative interior of $\sigma^{\vee}_{v}$. By properness of $\D$, and up to a change of $m$ by a strictly positive integer multiple, we can choose a homogeneous element $f\chi^m\in A[C_0,\D]$ of degree $m$ such that $(C_0)_f^m = C_0\backslash Z(f)$ contains $z$. Then $\D_f^m$ is of the form  
$$\sum_{z'\in (C_0)^m_f}Q_{z'}\cdot [z'],$$ 
where $Q_{z'}\subset N_{\QQ}$ is a polytope for every $z'\in (C_0)_f^m$ and $Q_z = \{v\}$. The set of colors of $\D_f^m$ is empty. By Step $1$ applied to the open subset $X(\D_f^m,\emptyset)\subset X(\D,\F)$, the germ $D'_{(z,v)}\subsetneq X(\D_f^m,\emptyset)$ corresponding to the colored datum $([(z, \QQ_{\geq 0}(v,1))], \emptyset)$ is of codimension one. Thus $D_{(z,v)}$ is of codimension one as the closure of $D'_{(z,v)}$ in $X(\D,\F)$. This proves the existence of the parametrization of the $G$-divisors.
\end{proof}

\begin{corollary} \label{clgroup} 
The class group ${\Cl}(X(\ES))$ is isomorphic to the abelian group 
$${\rm Cl}(C_{\ES})\oplus \bigoplus_{(z,v)\in\Vert(\ES)}\ZZ D_{(z,v)}\oplus\bigoplus_{\rho\in \Ray(\ES)}\ZZ\,D_{\rho}
\oplus \bigoplus_{\alpha\in S\setminus I}\ZZ\,D_{\alpha},$$
where $D_{\alpha}\subset X(\ES)$ is the color associated with $\alpha \in S\setminus I$, modulo the relations: 
$$[z] = \sum_{(z,v)\in \Vert(\ES)} \mu(v)\, D_{(z,v)} \text{ and }$$
$$\sum_{(z,v)\in \Vert(\ES)}\mu(v)\langle m, v\rangle D_{(z,v)} + 
\sum_{\rho\in\Ray(\ES)}\langle m,\rho\rangle D_{\rho} +
\sum_{\alpha\in S\setminus I}\langle m, \varrho(D_{\alpha})\rangle D_{\alpha}=0,$$  
where $m\in M$, $z\in C_{\ES}$, and $\mu(v) = \inf\{d\in \ZZ_{>0}\ |\ dv \in \NN \}$. 
\end{corollary}  

\begin{proof}
By \cite[Prop. 17.1]{Ti3}, every divisor $X(\ES)$ is linearly equivalent to a $B$-stable divisor. Hence, by Theorem \ref{theodiv}, we have a surjective homomorphism $\pi_{\ES}$ from the free abelian group
\begin{equation} \label{GammaES}
\Phi_{\ES}:=\bigoplus_{(z,v)\in\Vert(\ES)}\ZZ D_{(z,v)}\oplus\bigoplus_{\rho\in \Ray(\ES)}\ZZ\,D_{\rho}
\oplus \bigoplus_{\alpha\in S\setminus I}\ZZ\,D_{\alpha}
\end{equation}
onto ${\Cl}(X(\ES))$. 
The kernel of $\pi_{\ES}$ is formed by principal divisors associated with the $B$-eigenvectors of $k(X(\ES))$, i.e., by elements of the form ${\div}(f\chi^m):$  
$$\sum_{(z,v)\in\Vert(\ES)}v_{D_{(z,v)}}(f\chi^m)\, \cdot \, D_{(z,v)} + 
\sum_{\rho\in\Ray(\ES)}v_{D_{\rho}}(f\chi^m)\, \cdot \, D_{\rho} + \sum_{\alpha\in S\setminus I}v_{D_{\alpha}}(f\chi^m)\, \cdot \, D_{\alpha}$$
$$ = \sum_{(z,v)\in\Vert(\ES)}\mu(v)(\langle m, v\rangle + {\ord}_z f)D_{(z,v)}+  
\sum_{\rho\in\Ray(\ES)}\langle m, \rho\rangle D_{\rho} + 
\sum_{\alpha\in S\setminus I}\langle m, \varrho(D_{\alpha})\rangle D_{\alpha},$$
where $f\in k(C)^{\star}$ and $m\in M$. Let us now consider the surjective homomorphism from
$$\Cl(C_{\ES})\oplus\bigoplus_{(z,v)\in\Vert(\ES)}\ZZ\, D_{(z,v)}\oplus\bigoplus_{\rho\in \Ray(\ES)}\ZZ\,D_{\rho}
\oplus \bigoplus_{\alpha\in S\setminus I}\ZZ\,D_{\alpha}$$
onto $\Phi_{\ES}/\,\Ker\,\pi_{\ES}$ defined by
$$\sum_{z\in C_{\ES}}a_{z}\cdot [z] + \sum_{i}a_{i}D_{i}\mapsto \sum_{z\in C_{\ES}} a_z \left(\sum_{(z,v)\in \Vert(\ES)}\mu(v)[D_{(z,v)}] \right)
+\sum_{i}a_{i}[D_{i}],$$
where the $D_{i}$ represent the $B$-divisors on $X(\ES)$. Then one may check that the kernel of this homomorphism is exactly given by the relations stated above.
\end{proof}

\begin{example}
Returning to the example of the $\SL_3$-variety 
$$ X(\D)=\{xz-y^2=0\} \cap (\PP(-1,-1,-1,0,-1,-2) \; \backslash \; \PP(0,-1,-2))$$
considered in \S \ref{subsec22}, we can apply Corollary \ref{clgroup} to determine the class group of $X(\D)$. 
We obtain that $\Cl(X(\D))$ is the abelian group $\ZZ D_{(0,\frac{1}{2})} \oplus \ZZ D_\rho \oplus \ZZ D_\alpha$, where $\rho=\QQ_{\geq 0}$ and $D_\alpha$ is the unique color of $X(\D)$, modulo the following relations:
\begin{itemize}
\item $2 D_{(0,\frac{1}{2})}=0$; and
\item $m D_\rho+ 2 m D_\alpha=0$ for every $m \in \ZZ$.  
\end{itemize}
It follows that $\Cl(X(\D)) \cong \ZZ \oplus \ZZ/2 \ZZ$.
\end{example}

\subsubsection{}  \label{notY}
Let $Y(\ES)$ be the $K$-variety defined at the beginning of \S \ref{subsec22}. We denote by $\rho\mapsto \Gamma_{\rho}$,
$(z,v)\mapsto \Gamma_{(z,v)}$ the parametrization of the $K$-divisors on $Y(\ES)$ given by Theorem \ref{theodiv}. 
We recall that a normal variety is called \emph{factorial} if its class group is trivial.  
The next corollary gives a criterion of factoriality for the $G$-variety $X(\D)$. 

\begin{corollary} \label{cordiv}  
Let $(\D,\F)$ be a proper colored $\sigma$-polyhedral divisor on a dense open subset $C_0\subset C$. 
Then $X(\D)$ is factorial if and only if the two following conditions are satisfied.
\begin{enumerate}[(i)]
\item The equality ${\Cl}(Y(\D)) = \sum_{\Gamma_{\rho}\subset \Gamma}\mathbb{Z}[\Gamma_{\rho}]$ holds, where $\Gamma$ denotes the union of the $K$-divisors $\Gamma_{\rho}$ with $\rho$ satisfying $\varrho(\F)\cap \rho\neq \emptyset$.
\item For every $\alpha\in S\setminus I$, there exists $m_{\alpha}\in M$ and $f_\alpha\in k(C)^{\star}$ such that 
$$\langle m_{\alpha}, \varrho(D_{\alpha})\rangle  = 1,$$
and for all $\beta\in S\setminus (I\cup \{\alpha\})$, $(z,v)\in\Vert(\D)$, and $\rho\in\Ray(\D,\F)$:
$$\langle m_{\alpha}, \varrho(D_{\beta})\rangle  = \mu(v)(\langle m_{\alpha},v\rangle + {\ord}_z(f_{\alpha}))  = \langle m_{\alpha}, \rho\rangle  = 0.$$
\end{enumerate}
\end{corollary}

\begin{proof}
The decoloration morphism induces an isomorphism of varieties 
$$X(\D,\F) \backslash  \Gamma_0\simeq X(\D,\emptyset)\backslash \pi_{\dc}^{-1}(\Gamma_0),$$ 
where $\Gamma_0$ is the $G$-stable closed subset 
$$\left\{x\in X(\D,\F) \ \middle| \ Gx\subset \bigcup_{\alpha\in S\setminus I}D_{\alpha}\right\}.$$ 
As the codimension of $\Gamma_0$ in $X(\D,\F)$ is at least two, if $\Gamma'\subset X(\D,\emptyset)$ is the union of irreducible components of codimension one of $\pi_{\dc}^{-1}(\Gamma_0)$, then we have a group isomorphism 
$${\Cl}(X(\D,\F))\simeq {\Cl}\left(X(\D,\emptyset)\backslash \Gamma'\right).$$
By Proposition \ref{dec}, we can identify $X(\D,\emptyset)$ with $G/H\times^{K}Y(\D)$ and $\Gamma'$ with $G/H\times^{K}\Gamma$. 
Let us consider the chart $X_0 = q^{-1}(B\bar{x}_0)$, where $B\bar{x}_0$ is the open $B$-orbit of $G/P$ and 
$$q:G/H\times^{K}Y(\D) = G\times^{P}Y(\D)\rightarrow G/P$$
is the projection. The complement of the union of colors
$\bigcup_{\alpha\in S\setminus I} D_{\alpha}$ in $X(\D,\emptyset)\backslash \Gamma'$ is exactly
$$X_0\backslash (X_0 \cap \Gamma') \simeq B\bar{x}_0\times (Y(\D)\backslash \Gamma).$$
As $B\bar{x}_0$ is an affine space, by reiterating several times \cite[Prop. II.6.6]{Har}, we obtain
$${\Cl}\left(X_0\backslash \Gamma'\right)\simeq {\Cl}\left(Y(\D)\backslash \Gamma\right).$$ 
To sum up, we have an exact sequence:
$$\bigoplus_{\alpha\in S\setminus I}\mathbb{Z} D_{\alpha}\rightarrow  {\Cl}\left(X(\D,\emptyset)\backslash \Gamma'\right)\rightarrow {\Cl}\left(Y(\D)\backslash \Gamma\right)
\rightarrow 0,$$  
where the first arrow is induced by the surjective homomorphism from the group $\Phi_{\ES}$ defined by \eqref{GammaES} onto $\Cl(X(\D))$. Therefore ${\Cl}(X(\D)) = 0$ if and only if ${\Cl}\left(Y(\D)\backslash \Gamma\right) = 0$ and the first arrow above is zero. By Corollary \ref{clgroup}, this corresponds exactly to the conditions $(i)$ and $(ii)$. This proves the Corollary.  
\end{proof}

\begin{remark}
In general, the factoriality of $Y(\D)$ does not imply the one of $G\times^{P} Y(\D)$. 
For instance, we can consider $G = \SL_{2}$, $H$ the unipotent subgroup of upper triangular matrices, $C = \PP^{1}$, and $(\D,\emptyset)$ the colored $\QQ_{\geq 0}$-polyhedral divisor trivial on $\AA^{1}\subset \PP^{1}$. 
The $\SL_{2}$-variety $X(\D)$ identifies with $\AA^{1}\times {\Bl}_0(\AA^{2})$, where ${\Bl}_0(\AA^{2})$ is the bowing-up of $\AA^{2}$ at the origin.
We have ${\Cl}(X(\D))\simeq \mathbb{Z}$ whereas ${\Cl}(Y(\D)) = {\Cl}(\AA^{2}) = 0$.
\end{remark}

\subsubsection{} \label{Cartier}
As a by-product of Theorem \ref{theodiv}, we can refine the description of $B$-stable Cartier divisors of \cite{Tim00} to our setting.  
Let us start with a definition. 

\begin{definition} \label{deftheta}
Let $\ES = \{(\D^i,\F^i)\}_{i\in J}$ be a colored divisorial fan on $C$, and let $\D = \D^i\in\ES$ be a colored polyhedral divisor on $C_0$ with set of colors $\F=\F^i$. 
Recall that we denote by $\mathscr{C}(\D)$ the hypercone associated with $\D$ (see Remark \ref{link}). An \emph{integral linear function} on $\D$ is a map
$$\theta: \mathscr{C}(\D)\rightarrow \QQ$$
satisfying the following properties:
\begin{itemize}
\item[(i)] For every $z\in C_0$, there exists $m_z\in M$ and $\gamma_z\in\ZZ$ such that $\theta(z,u,l) = u(m_z) + l\gamma_z,$ for every $(u,l)\in \mathscr{C}(\D)_z$.
\end{itemize}
If $C_0=C$, then $\theta$ must satisfy the extra condition:
\begin{itemize}
\item[(ii)] We have $m:=m_z = m_{z'}$, for every $z,z'\in C$, and there exists $f\in k(C)^{\star}$ such that
$${\div} f = \sum_{z\in C}\gamma_z\cdot [z].$$ 
\end{itemize}
Let us denote by $\F_{\ES}$ the union of all the sets $\F^i$, where the $(\D^i,\F^i)$ run over $\ES$.
A \emph{colored integral piecewise linear function} on $\ES$ is a pair $(\theta, (r_{\alpha}))$, where $\theta$ is a function 
$$\theta :|\ES| \left(= \bigcup_{i \in J} \mathscr{C}(\D^i) \right) \rightarrow \QQ$$
such that the restriction $\theta_{|\mathscr{C}(\D^i)\cap\mathscr{C}(\D^{j})}$ is integral linear for every $i,j\in J$, and where $(r_{\alpha})$ is a sequence of integers with $\alpha$ running over the set of roots $\alpha\in S\setminus I$ such that $D_{\alpha}\not\in \F_{\ES}$. 

The pair $(\theta, (r_{\alpha}))$ is called \emph{principal} if $\theta$ satisfies (ii) and $r_{\alpha} = \langle m, \varrho(D_{\alpha})\rangle$. We denote respectively by $\PL(\ES)$ and $\Prin(\ES)$ the abelian groups (for the natural additive law) of colored integral piecewise linear functions of $\ES$ and of principal colored integral piecewise linear functions of $\ES$. If $\ES$ has a single element $\D$, then we will denote $\PL(\D)$ and $\Prin(\D)$ instead of $\PL(\ES)$ and $\Prin(\ES)$ respectively.

\end{definition}

As a direct consequence of \cite{Kno94}, \cite[\S 4]{Tim00}, \cite[\S 17]{Ti3}, and Theorem \ref{theodiv}, we obtain the next result.
See \cite[\S 3.1]{Bri89} for the spherical case and \cite[Cor. 3.19]{PS} for the case of normal $T$-varieties.

\begin{corollary} \label{corcartier}
With the notation above, if $(\theta, (r_{\alpha})) \in\PL(\ES)$, then
$$D_{\theta}: = \sum_{(z,v)\in\Vert(\ES)}\theta(z,\mu(v)(v,1))\cdot D_{(z,v)} +  
\sum_{\rho\in\Ray(\ES)}\theta(\cdot,\rho,0)\cdot D_{\rho}$$
$$ + \sum_{D_{\alpha}\in\F_{\ES}}\theta(\cdot,\varrho(D_{\alpha}),0)\cdot D_{\alpha}
+ \sum_{D_{\alpha}\not\in\F_{\ES}}r_{\alpha}\cdot D_{\alpha}$$
is a $B$-stable Cartier divisor on $X(\ES)$. More precisely, the map $\theta\mapsto D_{\theta}$ is an isomorphism between the group $\PL(\ES)$ and the group of $B$-stable Cartier divisors on $X(\ES)$, and there is a short exact sequence:
$$0\rightarrow \Prin(\ES)\rightarrow \PL(\ES)\rightarrow {\Pic}(X(\ES))\rightarrow 0,$$
where ${\Pic}(X(\ES))$ is the Picard group of $X(\ES)$.
\end{corollary}

\subsection{Canonical class and log-terminal singularities}  \label{subsec24}
In this section, we give an explicit representative of the canonical class for $X$ a normal horospherical $G$-variety of complexity one; see Theorem \ref{theocan}. From this, we deduce a criterion for $X$ to be $\QQ$-Gorenstein; see Corollary \ref{Goren}. Then we construct an explicit resolution of singularities of $X$; see Proposition \ref{res}. Finally, we obtain a criterion to determine whether the singularities of $X$ are log-terminal; see Theorem \ref{theolog}.

\subsubsection{}
The next result gives an explicit canonical divisor for any normal horospherical $G$-variety of complexity one; see \cite[\S 5]{Mos1} for the case of the embeddings of $\SL_2$, \cite{Bri93} for the case of the spherical varieties, and \cite[Cor. 4.25]{FZ}, \cite[Th. 3.21]{PS} for the case of normal $T$-varieties.  

\begin{theorem} \label{theocan}
Let $\ES$ be a colored divisorial fan on $C$. Then with the notation of \S \ref{RayVert} every canonical divisor on $X=X(\ES)$ is linearly equivalent to 
$$K_{X} = -\sum_{\rho\in\Ray(\ES)}D_{\rho}\;+ \sum_{(z,v)\in \Vert(\ES)}(\mu(v)b_z + \mu(v) -1)D_{(z,v)}
\;-\sum_{\alpha\in S\setminus I}a_{\alpha}D_{\alpha},$$ 
where $K_{C} = \sum_{z\in C}b_z\cdot [z]$ is a canonical divisor on $C$, $a_{\alpha} = \langle \sum_{\beta\in\Phi^{+}\setminus \Phi_{I}}\beta, \hat{\alpha}\rangle\geq 2$, $\Phi^+$ is the set of positive roots with respect to $(T,B)$, and $\Phi_I$ is the set of roots that are sums of elements of $I$.
\end{theorem}

\begin{proof}
Let us consider the union 
$$\Gamma' = \Gamma_0\cup {\Sing}(X)\subset X,$$ 
where $\Gamma_0$ is the biggest $G$-stable closed subset contained in the union of all the colors of $X$ and ${\Sing}(X)$ is the singular locus. Then $X(\ES') = X\backslash \Gamma'$ is smooth, quasi-toroidal and the set $\Vert(\ES)\coprod\Ray(\ES)$ identifies with the set $\Vert(\ES')\coprod\Ray(\ES')$.
Therefore, by Proposition \ref{dec}, we can assume without loss of generality that $X = G\times^{P}Y(\ES)$ is smooth and quasi-toroidal. 
Adapting the argument of the proof of \cite[Prop. 4.2 c)]{ST} to our setting, we obtain an isomorphism of $\mathcal{O}_{X}$-modules
$$\mathcal{O}_{X}(K_{X}) \simeq \mathcal{O}_{X}(D)\otimes q^{*}\mathcal{O}_{G/P}(K_{G/P}),$$
where $K_{G/P}$ is a canonical divisor on $G/P$ given by 
$$K_{G/P} = -\sum_{\alpha\in S\setminus I}a_{\alpha}D_{\alpha}$$ 
and the divisor $D$ is defined by
$$D = \sum_{(z,v)\in \Vert(\ES)}c_{(z,v)}D_{(z,v)} + \sum_{\rho\in\Ray(\ES)}c_{\rho}D_{\rho}$$
such that
$$K_{Y(\ES)} = \sum_{(z,v)\in \Vert(\ES)}c_{(z,v)}\Gamma_{(z,v)} + \sum_{\rho\in\Ray(\ES)}c_{\rho}\Gamma_{\rho}.$$
is a canonical divisor on $Y(\ES)$ (see the beginning of \S \ref{notY} for the notation). We conclude by \cite[Th. 3.21]{PS}.  
\end{proof}

\subsubsection{}
A normal $G$-variety $X$ is called \emph{$\QQ$-Gorenstein} if one (and thus any) canonical divisor $K_{X}$ is $\QQ$-Cartier.
The next corollary gives a combinatorial criterion for a normal horospherical $G$-variety of complexity one to be $\QQ$-Gorenstein; see \cite[Prop. 4.1]{Bri93} for the spherical case and \cite[Prop. 4.3]{LS13} for the case of normal $T$-varieties.

\begin{corollary} \label{Goren}
With the notation of \S\S \ref{RayVert} and \ref{Cartier}, the variety $X(\ES)$ is $\QQ$-Gorenstein if there exists $d\in\ZZ_{>0}$ and $\theta\in \PL(\ES)$ such that the following conditions are all satisfied.
\begin{enumerate}[(i)]
\item For every $\rho\in\Ray(\ES)$, we have $\theta(\cdot, \rho, 0) = -d$.
\item There exists a canonical divisor $K_{C} = \sum_{z\in C}b_z\cdot [z]$ on $C$ such that, for every $(z,v)\in\Vert(\ES)$, we have $\theta(z, \mu(v)v, 1) = d(\mu(v)b_z + \mu(v) - 1)$. 
\item For every $D_{\alpha}\in\F_{\ES}$, we have $\theta(\cdot,\varrho(D_{\alpha}),0) = -da_{\alpha}$. 
\end{enumerate}
\end{corollary}

\begin{proof}
It is a straightforward consequence of Corollary \ref{corcartier} and Theorem \ref{theocan}.
\end{proof}

\subsubsection{}
The remainder of this paper is dedicated to the study of log-terminal singularities of a simple $G$-model of $Z$.

Let $X$ be a $\QQ$-Gorenstein variety, let $\phi:X'\to X$ be a resolution of singularities, and let $d\in\ZZ_{>0}$ such that $dK_{X}$ is Cartier. Then the pull-back $\phi^{*}(dK_{X})$ is well-defined. The \emph{discrepancy} of $\phi$ is the $\QQ$-divisor
$$K_{X'} - \phi^{*}(K_{X}) := K_{X'} - \frac{1}{d}\phi^{*}(dK_{X}).$$
We say that $X$ has (purely) \emph{log-terminal} singularities if each coefficient of $K_{X'} - \phi^{*}(K_{X})$ is strictly bigger than $-1$. The property of having log-terminal singularities does not depend on the choice of the resolution of singularities $\phi$. More generally, if $D$ is a $\QQ$-divisor on $X$ such that $K_X+D$ is $\QQ$-Cartier, then we say that the pair $(X,D)$ is (purely) log-terminal if each coefficient of $K_{X'}-\phi^*(K_X+D)$ is strictly bigger than $-1$.

The next lemma will be useful to prove Theorem \ref{theolog}; see \cite[\S 3]{Kol97} for more details about log-terminal singularities and for a proof of the statement.

\begin{lemma} \label{lemmelog}
Let $\phi: X'\to X$ be a proper birational morphism between normal varieties. 
Let $D$ be a $\QQ$-divisor on $X$ such that $K_{X}+D$ is $\QQ$-Cartier, and let $D'$ be the $\QQ$-divisor on $X'$ defined by  
$$K_{X'}+D' = \phi^{*}(K_{X} + D).$$
Then $(X,D)$ is log-terminal if and only if $(X',D')$ is log-terminal and the coefficients of the prime divisors of $-D'$ (corresponding to exceptional divisors of $\phi$) are strictly bigger than $-1$.
\end{lemma}

\subsubsection{} 
In this subsection, we give an explicit method to construct a resolution of singularities of the $G$-model $X(\D)$ of $Z$.

Let $(\D,\F)$ be a proper colored $\sigma$-polyhedral divisor on a dense open subset $C_0\subset C$.
With the notation of \S \ref{subsec13}, we denote 
$$\tilde{Y}(\D) = {\Spec}_{\mathcal{O}_C}\left(\bigoplus_{m\in\sigma^{\vee}\cap M}\mathcal{O}_{C}(\D(m))\chi^m\right).$$
Then the natural morphism $\tilde{Y}(\D)\to Y(\D)$ induces a partial desingularization
$$G\times^{P}\tilde{Y}(\D)\rightarrow G\times^{P}Y(\D)=X(\D,\emptyset);$$
see \cite[Th. 3.1 (ii)]{AH} and \cite[Construction 2.8]{LS13b}. 
The $G$-variety $G\times^{P}\tilde{Y}(\D)$ identifies with $X(\ES_{{\tor}})$, where $\ES_{{\tor}}$ is the colored divisorial fan $\ES_{{\tor}} = \{(\D_{|C_i},\emptyset)\}_{i\in J};$
the sequence $(C_i)_{i\in J}$ forming a covering of $C_0$ by affine open subsets. 
Moreover, if we consider a divisorial fan $\overline{\Sigma}$ that refines $\Sigma_{\tor}$ and such that for all colored polyhedral divisor $\D \in \overline{\Sigma}$ and $z \in C$ with $\mathcal{C}(\D)_z \neq \emptyset$, the polyhedral cone $\mathcal{C}(\D)_z$ is regular (i.e., a cone generated by a subset of a basis of the lattice $N\times\ZZ$), then $X(\bar{\ES})\rightarrow X(\ES_{\tor})$ is a resolution of singularities. 
The next result is a straightforward consequence of this discussion.

\begin{proposition} \label{res}
The morphism
$$\phi:\; X(\bar{\ES})\rightarrow X(\ES_{\tor}) \rightarrow X(\D,\emptyset) \rightarrow X(\D,\F)$$
obtained by composing the decoloration morphism of \S \ref{subsec22} with the morphisms defined above is a resolution of singularities of $X(\D,\F)$. 
Moreover, with the notation of \S \ref{RayVert}, the exceptional divisors of $\phi$ correspond to the subsets $\Ray(\overline \Sigma) \backslash \Ray(\D)$ and $\Vert(\overline \Sigma) \backslash \Vert(\D)$.
\end{proposition}

\subsubsection{}
The next statement gives a characterization of the normal horospherical $G$-varieties of complexity one having log-terminal singularities. 
See \cite[Th. 4.1]{Bri93} for the spherical case and \cite[Th. 4.9]{LS13} for the case of normal $T$-varieties.

\begin{theorem} \label{theolog}   
Let $(\D,\F)$ be a proper colored $\sigma$-polyhedral divisor on a dense open subset $C_0\subset C$.
We suppose that $X(\D)$ is $\QQ$-Gorenstein. Then $X(\D)$ has log-terminal singularities if and only if one of the following assertions holds.
\begin{enumerate}[(i)]
\item The curve $C_0$ is affine.
\item The curve $C_0$ is the projective line $\PP^{1}$ and $\sum_{z\in C_0}\left(1 - \frac{1}{\mu_z}\right)< 2$, where for every $z\in C_0$ we denote $\mu_z := \max\{\mu(v)\ |\ v\in\Delta_z(0)\}$ and $\mu(v):=\inf\{d\in \ZZ_{>0}\ |\ dv \in \NN \}$. 
\end{enumerate} 
\end{theorem}

\begin{proof}
If $C_0$ is affine, then by the proof of Theorem \ref{crit1}, we obtain that $X(\D)$ is covered by \'etale open subsets of horospherical embeddings.
Hence, by \cite[Th. 4.1]{Bri93}, $X(\D)$ has log terminal singularities. Therefore, we can assume that $C_0=C$ is projective. 
Let us consider a canonical divisor $K_X$ as in Theorem \ref{theocan}, let $d \in \ZZ_{>0}$ be such that $d K_X$ is a Cartier divisor, and let $\theta\in\PL(\D)$ such that $dK_{X} = D_{\theta}$.
By Corollary \ref{corcartier}, we know that the restriction of $d K_X$ on the open subset  
$$X_{1} := X(\D)\backslash \bigcup_{D_{\alpha}\not\in\F_{\ES}}D_{\alpha}$$
is a principal divisor. Moreover, since every color $D_\alpha$ satisfying $D_\alpha \notin \F_{\ES}$ does not contain a $G$-orbit of $X(\D)$, the open subset $\psi^{-1}(X_{1})$ intersects each exceptional divisor of the partial desingularization $\psi: X(\ES_{\tor}) \rightarrow X(\D)$ given by Proposition \ref{res}. Therefore we can replace $X$ by $X_1$ and suppose that $d K_X$ is principal. It follows that $\psi^*(dK_{X})$ is the principal divisor of a homogeneous element $f\chi^m\in A_M$ of degree $m$ considered as a rational function of $X(\ES_{{\tor}})$.  Hence, we have the equality
$$-D':= K_{X(\ES_{{\tor}})} - \psi^* K_{X(\D)} = \sum_{\rho\not\in\Ray(\D)}(-1-\langle m,\rho\rangle )D_{\rho}.$$ 
By Lemma \ref{lemmelog}, $X(\D)$ has log-terminal singularities if and only if $-D'$ has its coefficients strictly bigger than $-1$ and $(X(\ES_{{\tor}}), D')$ has log-terminal singularities. Thus, by the same argument as in \cite[\S 4]{LS13} (see the sketch of proof before \cite[Th. 4.9]{LS13}), we know that $X(\D)$ has log-terminal singularities if and only if $\langle m, \rho\rangle < 0$, for every $\rho\not\in\Ray(\D)$. 

Let $\rho\subset\sigma$ be an extremal ray such that $\rho\not\in\Ray(\D)$, then $\rho \cap \varrho(\F) \neq \emptyset$ or $\rho \cap \deg \D \neq \emptyset$. Let us suppose that $\rho\cap \varrho(\F)\neq \emptyset$. Then there exists 
$D_{\alpha}\in\F$ and $\lambda\in\QQ_{>0}$ such that $\rho = \lambda\varrho(D_{\alpha})$. Hence
$$\langle m,\rho\rangle  = \lambda \langle m,\varrho(D_{\alpha})\rangle = -\lambda d a_{\alpha} < 0.$$
Let us now suppose that $\rho \cap \deg \D \neq \emptyset$. Then $\rho = \lambda v$ for a vertex $v\in{\deg}\D$ and for some $\lambda\in\QQ_{>0}$. 
Let $(v_z)_{z\in C}$ be a sequence of elements of $\Delta_z(0)$ such that $v = \sum_{z\in C}v_z$. 
As the coefficients of $K_X$ and $\frac{1}{d} \div f \chi^m$ at the prime divisor $D_{(z,v_z)}$ corresponding to any $(z,v_z) \in \Vert(\D) $ are the same, we have equalities     
$$\mu(v_z)b_z+\mu(v_z)-1=\frac{1}{d} \mu(v_z) (\left\langle m, v_z \right\rangle+{\ord}_z f),$$
where $K_C=\sum_{z \in C} b_z \cdot [z]$ is a canonical divisor on $C$. Since $C$ is projective, we have $\deg (\div f)=0$. Hence, summing over $C$ on both sides gives the equality
$${\deg} K_{C} + \sum_{z\in C}\left(1 - \frac{1}{\mu(v_z)}\right)=\frac{1}{d}\langle m, v\rangle.$$
As $\deg \D \subset \sigma$, we conclude that $X(\D)$ has log-terminal singularities if and only if the condition $(ii)$ is satisfied.

\end{proof}

\end{document}